\documentclass{article}
\usepackage{graphicx,color}

\usepackage{stmaryrd,amsmath,amsfonts,amssymb,amsthm,mathrsfs,amsopn}
\usepackage{latexsym}

\usepackage[numbers]{natbib}

\usepackage{hyperref}

\usepackage[usenames,dvipsnames,x11names,table]{xcolor}

\numberwithin{equation}{section}

\def\R{\mathbb R}
\def\N{\mathbb N}

\def\e{\varepsilon}

\def\de{\partial}

\def\00{{\bf 0}}

\newcommand\res{\mathop{\hbox{\vrule height 7pt width .3pt depth 0pt \vrule height .3pt width 5pt depth 0pt}}\nolimits}

\def\FF{\mathcal{F}}

\newcommand{\ang}[1]{\langle#1\rangle}

\theoremstyle{plain}
\newtheorem{theorem}{Theorem} [section]
\newtheorem{corollary}[theorem]{Corollary}
\newtheorem{lemma}[theorem]{Lemma}
\newtheorem{proposition}[theorem]{Proposition}
\theoremstyle{definition}

\newtheorem{remark}[theorem]{Remark}
\newtheorem*{ack}{Acknowledgements}

\renewcommand{\epsilon}{\varepsilon}
\renewcommand{\phi}{\varphi}
\renewcommand{\bar}{\overline}

\title{Anisotropic min-max via phase transitions}
\author{Antonio De Rosa and Alessandro Pigati}

\date{}

\begin{document}

\maketitle

\begin{abstract}
We develop a PDE-based approach to the min-max construction of nontrivial integer rectifiable varifolds that are stationary with respect to anisotropic surface energies on closed Riemannian manifolds, in codimension one. Specifically, we study the anisotropic analogue of the Allen--Cahn energy and establish a Modica-type gradient bound for its critical points. Using this in conjunction with certain estimates for stable solutions, we then prove that the energy densities of stable or bounded-Morse-index critical points of its rescalings 
concentrate along an integer rectifiable varifold that is stationary for the underlying anisotropic integrand. As a consequence, we construct a (possibly singular) anisotropic min-max hypersurface via Allen--Cahn,
obtaining an analogue of the result of Hutchinson--Tonegawa in the anisotropic setting.
\end{abstract}

\section{Introduction}
The Allen--Cahn equation has become a central tool in the study of minimal hypersurfaces and related variational problems in geometric analysis. Originally introduced by Allen and Cahn~\cite{AllenCahn1977} as a diffuse-interface model for phase transitions, this reaction-diffusion equation reads
\begin{equation*}\label{flow}
    \partial_t u - \Delta u + \frac{W'(u)}{\varepsilon^2} = 0.
\end{equation*}
Here $W$ denotes a double-well potential with global minima at $\pm 1$, the prototypical example being $W(u) = \frac{1}{4}(1 - u^2)^2$. Modica and Mortola~\cite{ModicaMortola1977} observed that the associated energy functional
\begin{equation}\label{eq:energy}
    \int \left[\varepsilon \frac{|\nabla u|^2}{2} + \frac{W(u)}{\varepsilon}\right],
\end{equation}
$\Gamma$-converges, as the interfacial width parameter $\varepsilon \to 0$, to the area functional of sets of finite perimeter. This observation revealed a deep connection between the Allen--Cahn variational framework and the theory of minimal hypersurfaces in geometric measure theory, as originally suggested by De Giorgi. 

In particular, critical points of \eqref{eq:energy}, that is, functions $u$ satisfying
\begin{equation*}
    -\varepsilon \Delta u + \frac{W'(u)}{\varepsilon} = 0,
\end{equation*}
have been shown to concentrate on minimal hypersurfaces as $\e \to 0$. This phenomenon dates back to the foundational work of Modica~\cite{Modica1987} and Sternberg~\cite{Sternberg1988}, and has since developed into a well-established field. We refer the reader to the surveys~\cite{Pacard2012,Savin2009,Tonegawa2008} for an overview.

Building on this idea and on the works of Ilmanen \cite{Ilmanen}, Hutchinson and Tonegawa~\cite{HutchinsonTonegawa2000}, Tonegawa~\cite{Tonegawa2005}, Tonegawa and Wickramasekera~\cite{TonegawaWickramasekera2012}, and Wickramasekera~\cite{Wickramasekera2014}, Guaraco~\cite{Guaraco2018} established an existence theorem for closed minimal hypersurfaces in arbitrary closed Riemannian manifolds by adapting the mountain-pass construction to the Allen--Cahn functional~\eqref{eq:energy}. His work provided a PDE-based alternative to the Almgren--Pitts min-max theory from~\cite{Almgren1965,Pitts1981}. The Almgren--Pitts theory has had a profound impact on geometric analysis and was later refined by Marques and Neves~\cite{MarquesNeves2014,MarquesNeves2017} in their proof of the Willmore conjecture, 
leading also to the recent resolution of Yau’s conjecture by Song~\cite{Song2023}.
On the other hand, the PDE-based min-max theory has since been extended to a variety of settings
and proved capable of achieving more refined results which predated analogues in the Almgren--Pitts framework, such as the multiplicity-one result of~\cite{ChodoshMantoulidis2023}. Collectively, these developments have established the Allen--Cahn framework as a flexible and powerful analytic alternative to the Almgren--Pitts one, allowing geometric-measure-theoretic ideas to be realized within a purely elliptic PDE setting,
with promising analogues in higher codimension:
see \cite{LRGamma,JS,ABO,LR, BBO, SternThesis, ChengThesis} for the Ginzburg--Landau model with no magnetic field and \cite{PigSte, PPSGamma} for the abelian Higgs model
(as well as \cite{SternS1,PPS} for slightly different settings).

Parallel to these isotropic developments, a significant body of work has emerged on minimal surfaces with respect to anisotropic surface energies, which naturally arise in crystalline surface tension models, capillarity problems, and in the modeling of interfaces between distinct materials. We refer to the survey~\cite{DeR-survey} for an overview of the theory of anisotropic minimal surfaces. The anisotropic area functional generalizes the classical area by introducing a convex, one-homogeneous integrand depending on both position and normal direction. While existence and regularity of minimizers for anisotropic energies are by now well understood~\cite{Alm68,Alm76,DLDRG,DPDRG2,DR}, a satisfactory existence and regularity theory for stationary solutions remains incomplete~\cite{allard1986integrality,DRR,DRT}, due largely to the absence of a monotonicity formula for density ratios~\cite{characterizationareaallard}. Allard~\cite{AllardStab} conjectured the existence of closed anisotropic minimal hypersurfaces in closed Riemannian manifolds, but an anisotropic counterpart of the Almgren--Pitts min-max theory had remained elusive until recently. In~\cite{DPDRminmax,DePhilippisDeRosaLi2025}  De Philippis, the first-named author, and Li have now resolved this problem by developing an anisotropic version of the Almgren--Pitts theory, proving the existence of closed anisotropic minimal hypersurfaces with essentially optimal regularity in any closed Riemannian manifold. 

The goal of the present paper is to provide a PDE-based alternative to the latter approach to the anisotropic min-max construction. More precisely, we develop an anisotropic version of the isotropic Allen--Cahn min-max construction by Guaraco \cite{Guaraco2018}. 
Let $(M^n,g)$ be a closed Riemannian manifold with $n \ge 2$. We consider a smooth even function
\[
    F : UT(M) \to (0,\infty), \quad F(x,v) = F(x,-v),
\]
defined on the unit tangent bundle, and extend $F$ $1$-homogeneously to the tangent bundle:
\[
    F : TM \to [0,\infty), \quad F(x,\lambda v) = \lambda F(x,v)\text{ for all } \lambda \ge 0,\ (x,v) \in TM.
\]
We assume that $F$ is uniformly convex in directions orthogonal to the radial one, namely
\[
    D^2F_x(v)[w,w]\ge \lambda|w|^2\quad\text{for }v\in UT_x(M)\text{ and }w\in T_xM,\ w\perp v,
\]
for all $x \in M$, where $F_x$ denotes the restriction of $F$ to $T_x M$.  
The anisotropic surface area of an embedded hypersurface $\Sigma^{n-1} \subset M^n$ is defined by
\[
    \mathcal{F}(\Sigma) := \int_\Sigma F(x,\nu_x)\, d\mathcal{H}^{n-1}_g(x),
\]
and the anisotropic Allen--Cahn energy of a function $u:M\to \R$ by
\[
    E_\varepsilon(u) := \int_M \left[ \varepsilon \frac{F(x,\nabla u(x))^2}{2} + \frac{W(u(x))}{\varepsilon} \right] \,d\operatorname{vol}_g(x).
\]
All the relevant assumptions on the double-well potential $W$ will be recalled in \eqref{dw1}--\eqref{dw2} in Section~\ref{sec:Prel}. The energy functional $E_\e$ has been recently investigated by Cicalese, Nagase, and Pisante \cite{CicaleseNagasePisante2010}, and the associated parabolic equation has been the subject of an extensive literature \cite{ElliottSchatzle1996,ElliottSchatzle1997,GigaOhtsukaSchatzle2006,Laux2020,MatanoMoriNara2019}.  Bouchitté \cite{Bouchitte1990} proved in a more general framework that $E_\varepsilon$ $\Gamma$-converges to $c_W \cdot \mathcal{F}$, where $c_W := \int_{-1}^1 \sqrt{2W}$. We provide a shorter proof of this $\Gamma$-convergence in Theorem~\ref{thm:GammaConv}, tailored for $E_\e$. We also remark that in recent years there has been a growing literature on the $\Gamma$-convergence to anisotropic perimeters of heterogeneous variants of the Allen--Cahn energy functional in a periodic medium \cite{ChoksiFonsecaLinVenkatraman2022,CristoferiFonsecaHagertyPopovici2019,FeldmanMorfe2023}. 

On the contrary, nothing was known in the literature for the convergence of stable (or bounded-Morse-index) critical points for $E_\e$ to anisotropic minimal hypersurfaces, as $\e \to 0$;
actually, as explained in the body of the paper, in order to have a $C^2$ functional
we need to perturb $F$ slightly in the variational construction, although ultimately we still obtain critical points for the original $E_\epsilon$. Establishing this convergence is a crucial step to extend the isotropic min-max construction of Guaraco~\cite{Guaraco2018} to the anisotropic Allen--Cahn framework. The main obstruction is the lack of a suitable monotonicity formula in the anisotropic setting \cite{characterizationareaallard}, which makes it a challenging task to prove that the limiting varifold has no diffuse part, and in particular that it is rectifiable.
For the isotropic Allen--Cahn functional, the monotonicity formula is deduced by a celebrated gradient bound proved by Modica \cite{Modica}. Our first  main contribution is the following anisotropic version of Modica's bound  (see Theorem~\ref{thm:modicabound} for the details).

\begin{theorem}\label{thminf:modicabound}
    Let $u : M \to [-1,1]$ be a critical point of $E_\varepsilon$. Then
    \begin{equation}\label{eq:modicabound}
        \frac{\varepsilon F^2(\nabla u)}{2} \le \frac{W(u)}{\varepsilon} + C(M^n,g,F).
    \end{equation}
\end{theorem}

Compared to the isotropic counterpart, its proof is much more delicate
due to the appearance of additional terms when the Euler--Lagrange equation is differentiated, which need to be controlled by exploiting convexity and homogeneity of $F$ at various places. An additional challenge is
the lack of effective bounds on the Hessian $D^2u$ away from $\{\nabla u=0\}$, as globally $u$ is guaranteed to belong only to $C^{1,\alpha}$ for some $\alpha\in(0,1)$; terms involving the Hessian naturally appear in the setting of inherently non-autonomous integrands $F$ on closed manifolds.

Theorem~\ref{thminf:modicabound} is somewhat surprising: in the isotropic Allen--Cahn case, the validity of~\eqref{eq:modicabound} implies the monotonicity formula \cite{Ilmanen,HutchinsonTonegawa2000}, which is known to fail in the anisotropic setting~\cite{characterizationareaallard}. Hence, while~\eqref{eq:modicabound} remains valid, it no longer entails monotonicity.
Nonetheless, Theorem~\ref{thminf:modicabound} serves as a key ingredient in proving our main result.

\begin{theorem}\label{thminf:min-max}
    There exists a nontrivial $F$-stationary integral $(n-1)$-varifold $V$ in $M$, whose weight $c_W\|V\|$ arises as the limit, as $\varepsilon \to 0$, of the energy densities \[
e_\varepsilon(u_\varepsilon)
:=
\frac{\varepsilon}{2} F(\nabla u_\varepsilon)^2
+\frac{1}{\varepsilon} W(u_\varepsilon)
\]
of suitable critical points $u_\varepsilon$ of  $E_\e$, constructed via a min-max procedure. 
\end{theorem}

\begin{remark}
    While we exhibit such $V$ from the simplest mountain-pass construction, the analysis contained
    in the present paper applies to any min-max scheme. We conjecture that, in fact, the energy density of \emph{any} sequence $(u_\e)$ of critical points with bounded energy concentrates along an integral $F$-stationary varifold as $\e\to0$, up to a subsequence, regardless of stability conditions.
\end{remark}

While we do not attempt to develop any regularity theory beyond integrality of $V$,
we hope to do so in a future work, exploiting once again the stability of $u_\e$.
It would also be interesting to investigate analogues in codimension two or higher,
e.g.\ by devising appropriate anisotropic versions of \cite{PigSte}.

Below, we briefly discuss the proof strategy of Theorem~\ref{thminf:min-max}, which is obtained combining Proposition \ref{prop:bounds} with Theorem~\ref{varifold.constr} and Theorem~\ref{thm:int}.

First of all, we  construct a  family of min-max solutions $(u_\e)$ satisfying uniform (in $\e$) lower and upper bounds of $E_\e(u_\e)$ and with Morse index $\le 1$ (up to a regularizing approximation of $F$); see Proposition \ref{prop:bounds}. Hence, the energy densities
of $u_\e$ will subconverge in the sense of measures to a finite Radon
measure $\mu$ on $M$.

To encode the geometry of the level sets of $u_\varepsilon$, we introduce the
$(n-1)$-varifolds $\tilde V_\varepsilon$  in \eqref{def:average}, which heuristically are weighted averages of the level sets of $u_\e$.
The weight measure $\|\tilde V_\varepsilon\|(M)$ can be bounded above (up to a constant)  with $E(u_\e)$, and hence we can extract a subsequential limit $\tilde V_0$ (as $\e \to 0$) with $\|\tilde V_0\|
\leq C\mu$. Moreover, using the Morse index bound of (the approximations of) $u_\e$, we are also able to deduce a local bound on the isotropic first variation of $\tilde V_0$ away from a finite set $\mathcal{S}$: see Corollary \ref{firstvar.bd.cor}. To this aim, we derive a diffuse version of the stability inequality for anisotropic minimal hypersurfaces, stated in Theorem~\ref{stab.ineq}. This result parallels those obtained in the isotropic case by Padilla and Tonegawa~\cite{PadillaTonegawa1998} and Tonegawa~\cite{Tonegawa2005}, although in our setting the PDE degenerates at points where $\nabla u_\e = 0$, as $F^2$ is not $C^2$ at the zero section (unless $F$ is a quadratic norm). To handle this, we perform our estimates first for suitable smooth approximations of $F^2$ (the same ones used to regularize $E_\epsilon$) and then pass to the limit.

By means of the Modica-type estimate \eqref{eq:modicabound} in Theorem \ref{thminf:modicabound} and
another application of stability, we can also control the weight $\|\tilde V_0\|$ from below with $\mu$ on $M\setminus \mathcal{S}$ and, more importantly, we can prove that $\tilde V_0$ and $\mu$ are rectifiable: see Theorem \ref{rect}. A crucial ingredient in the proof of Theorem  \ref{rect} is to show that the upper density $\Theta^{n-1,*}(\|\tilde V_0\|,x)>0$ for $\|\tilde V_0\|$-a.e.\ $x$.
This, combined with the local bound on the isotropic first variation of $\tilde V_0$, provides the rectifiability of $\tilde V_0$ and hence of $\mu$ on $M\setminus \mathcal{S}$, by Allard's rectifiability theorem \cite[Section 5]{allard1972first}.

Our next step involves extracting an $F$-stationary varifold $V$ and, as a byproduct of the analysis,
showing that $\mu$ is a rectifiable measure on all $M$, as stated in Theorem \ref{varifold.constr}. This is achieved by analyzing the stress-energy tensor
$T_\varepsilon$  of $u_\varepsilon$ defined in \eqref{strtens}, and proving its convergence as $\e \to 0$ to a tensor-valued measure
$T_0$ whose divergence is controlled. A careful analysis of the directions of invariance of the tangent measures of $T_0$ provides the desired rectifiability of $\mu$, with an argument in the spirit of \cite{de2018rectifiability}.
Also, writing $d\mu=\theta\,d(\mathcal{H}^{n-1}\res\Sigma)$, where $\Sigma\subset M$ is a rectifiable Borel set  with $\sigma$-finite $\mathcal{H}^{n-1}$ measure and $\theta:\Sigma\to(0,\infty)$, we show that the $(n-1)$-dimensional varifold
    $$dV(x,\nu):=\frac{\theta(x)}{F(x,\nu_x)} \delta_{\nu_x}(\nu)\otimes d(\mathcal{H}^{n-1}\res\Sigma)(x)$$
    is rectifiable and $F$-stationary, where $\nu_x\perp T_x\Sigma$ is the unit normal.

    Our candidate  varifold claimed in Theorem \ref{thminf:min-max} is $c_W^{-1}V$: we are just left to prove that $\frac{\theta(x)}{c_W F(x,\nu_x)}$
is an integer for $\mathcal{H}^{n-1}$-a.e.\ $x\in\Sigma$.
This is done by a one-dimensional slicing argument in the normal directions to $\Sigma$. Fix
a generic point $x_0\in\Sigma$ and assume
$\nu_x = e_n$ in a chart. On small cylinders of the form
$B^{n-1}_\epsilon(x_0)\times(-r,r)$, one rescales the functions $u_\varepsilon$ to obtain
limit profiles defined on Euclidean cylinders with almost flat metric and
almost autonomous integrand $F_{x_0}$.
Using the convergence of stress-energy tensors, we show in \eqref{par.to.zero} the smallness of the tangential gradient
  $\partial_{x_i}u_\varepsilon$ for $i<n$ in an $L^2$ sense, that is, the fact that $u_\e$ is almost one-dimensional.
  Such smallness, combined with the Modica-type inequality \eqref{eq:modicabound}, is then proved to imply 
 that $E_\e(u_\e)$ is close to ($F(x_0,e_n)$ times) a multiple of the energy $c_W=\int_{-1}^1\sqrt{2W}$ of the heteroclinical solution, unless $u_\e$ is essentially constant: see Lemma \ref{one-dim}. From this, we deduce that along almost every normal line  the energy
of $u_\varepsilon$ on that line converges to a multiple of
$c_W F(x_0,e_n)$.
This in turn implies that $\frac{\theta(x_0)}{c_W F(x,\nu_{x_0})}$ is an integer, as asserted in Theorem~\ref{thm:int}.

\begin{ack}
    Antonio De Rosa was funded by 
    the European Research Council (ERC), through StG ``ANGEVA,'' project number: 101076411. Alessandro Pigati was funded by 
    the European Research Council (ERC), through StG ``MAGNETIC,'' project number: 101165368.
Views and opinions expressed are however those of the authors only and do not necessarily reflect those of the European Union or the European Research Council. Neither the European Union nor the granting authority can be held responsible for them.

\end{ack}

\section{Preliminaries}\label{sec:preliminaries}

\subsection{Terminology}
Throughout the paper, we fix a smooth closed (i.e., compact and without boundary) \(n\)-dimensional Riemannian manifold $(M^n, g)$ with \(n\ge 2\).  We note that the fixed background Riemannian metric plays essentially no role in what follows, since it can be absorbed into the integrand. It is however useful in order to fix a background volume form and to identify \(n\)-dimensional planes with their normals.

We assume the reader to be familiar with the standard notions in geometric measure theory~\cite{Simon}, and we adopt the following notation and conventions.
\begin{itemize}
\item $UT(M)$:  the unit tangent bundle of $M$, namely,
  \[
    UT(M) := \{(x, v) \in TM: \|v\|_{g_x} = 1\}.
  \]
\item $G_{n-1}(M)$:  the unoriented hyperplane bundle of $M$, namely,
  \[
    G_{n-1}(M) := \{(x, T): x \in M,\ T \text{ is an $(n-1)$-dimensional linear subspace of } T_x M\}.
  \]
 By means of the background metric, we can identify \(G_{n-1}(M)\) with \( UT(M)/\sim\), where we have set the equivalence relation \((x,v) \sim (x, -v)\).
\item $\partial^* E$: the reduced boundary of a Caccioppoli set $E$.
\item $\mathcal{V}(M) = \mathcal{V}_{n-1}(M^n)$: the space of $(n-1)$-varifolds on $M$, namely, the space of nonnegative Radon measure on $G_{n-1}(M)$.
\item $\llbracket K \rrbracket$: 
the integral varifold of multiplicity one associated with an $(n-1)$-rectifiable set $K$ with $\mathcal{H}^{n-1}(K)<\infty$.
\item $\|V\|$: the weight of $V \in \mathcal{V}(M)$, i.e., the Radon measure on $M$ associated with $V$.
\end{itemize}

\subsection{Anisotropic energies}

In the sequel, an \emph{anisotropic integrand} will be a smooth function $F: G_{n-1}(M) \to (0,\infty)$. Note that, according to the identification above, we can consider it as a function \(F: UT(M)\to (0,\infty)\), which is an even function in the second variable. When no confusion arises, we will often switch between these viewpoints without further comment. Also, 
for the fixed background metric $g$, we will just use the symbol $|v|$ for $|v|_{g_x}=\sqrt{g_x(v,v)}$ and $\mathcal{H}^{n-1}$ for $\mathcal{H}^{n-1}_g$. It will also be useful to extend \(F\) $1$-homogeneously in the second variable as
\[
F(x, v)=|v|F\left(x, \frac{v}{|v|}\right),\quad F(x,0)=0.
\]
We will always assume this extension has been made whenever we consider derivatives of \(F\).

We assume that
$F$ is uniformly convex in directions orthogonal to the radial direction, namely
$$D^2F_x(v)[w,w]\ge \lambda|w|^2\quad\text{for }v\in UT_x(M)\text{ and }w\in T_xM,\ w\perp v$$
at all points $x\in M$, where $F_x$ denotes the restriction of $F$ to $T_xM$.
Note that, by $1$-homogeneity of $F$, the uniform convexity assumption gives
\begin{equation}\label{F.cvx}
D^2F_x(v)[w,w]\ge \lambda\frac{|w|^2}{|v|}\quad\text{for }v\in T_xM\setminus\{0\}\text{ and }w\in T_xM,\ w\perp v
\end{equation}
and the $2$-homogeneous function $F^2$ is automatically uniformly convex, namely
\begin{equation}\label{F.square.cvx}
D^2(F^2_x)\ge2\lambda g_x\quad\text{on }T_xM\setminus\{0\}
\end{equation}
in the sense of quadratic forms at each point $x\in M$, up to possibly decreasing $\lambda>0$.
Possibly further decreasing $\lambda \in (0,1)$, we can also assume that
\begin{equation}\label{bound}
    \lambda|v|\leq F_x(v)\leq \frac{|v|}{\lambda} \quad\text{for all }v\in T_xM.
\end{equation}

 For a finite perimeter set $E$, we denote the \emph{anisotropic perimeter} of $E$ by
$$\mathcal{F}(E):=\int_{\partial^*E} F(x,\nu_x)\,d\mathcal{H}^{n-1}(x),$$
where here $\nu_x$ denotes the exterior measure-theoretic normal of $E$. Actually, the assumption that $F$ is even implies that the sign of the unit normal $\nu_x$ is irrelevant.

For a varifold $V \in \mathcal{V}(M)$, we define its \emph{$F$-anisotropic energy}, and respectively its localized energy to a Borel subset \(U\subset M\), as
\[
  \FF(V) :=  \int_{G_{n-1} (M)} F(x, T) \,dV(x, T), \quad \mbox{and respectively} \quad \FF(V;U) :=  \int_{G_{n-1} (U)} F(x, T) \, dV(x, T),
\]
where \(G_{n-1} (U)\) is the restriction of the Grassmannian bundle to \(U\).
Note that, by the identification of the \((n-1)\)-dimensional Grassmanian with the unit sphere, we can equivalently think an \((n-1)\)-dimensional varifold as a measure on \(UT(M)\) invariant under reflection $(x,v)\mapsto(x,-v)$. In that case we will write
\[
  \FF(V) =  \int_{UT (M)} F(x, \nu) \, dV(x, \nu).
\]

The \emph{first variation} of the $F$-anisotropic energy is defined as
\[
  \delta_F V(X) := \frac{d}{dt}\Bigl |_{t = 0} \FF\left((\varphi_t)_\# V\right)
\]
where $V \in \mathcal{V}(M)$, $X\in C^1(M,TM)$, and $\varphi_t$ is the flow of $X$ (i.e., $\frac{d\varphi_t}{dt} = X(\varphi_t)$ and $\varphi_0 = \operatorname{id}_M$).  Referring to~\cite{de2018rectifiability} for the general expression, here we record that when \(M=\mathbb R^{n+1}\) and \(F(x,\nu)\) is autonomous, i.e., independent of the spatial variable $x$ (so that we can view it as an even function $\R^{n+1}\to\R$), we have the following formula:
\begin{align}
\label{e:firstvariation}
\begin{aligned}
 \delta_{F} V(X)&=\int_{UT (M)} [F (\nu)\operatorname{div} X-\langle DF(\nu), DX^\top \nu \rangle]\, dV(x,\nu)\\
 &=\int_{UT (M)} \langle F (\nu)I-\nu\otimes DF(\nu),DX \rangle\, dV(x,\nu),
\end{aligned}
\end{align}
where \(DX^\top \) is the transpose of \(DX\). For the general formula of $\delta_{F} V(X)$, one needs to add a term which depends on \(D_xF\) (see \cite{de2018rectifiability}). 
Back to the case of closed $M$, a varifold $V \in \mathcal{V}(M)$ is said to have \emph{bounded $F$-anisotropic first variation} if there exists $C > 0$ such that for all $X\in C^1(M,TM)$
\[
  |\delta_{F} V(X)| \leq C \|X\|_{L^\infty}.
\]
Equivalently, an $(n-1)$-varifold $V \in \mathcal{V}(M)$ has bounded $F$-anisotropic first variation if \(\delta_F V\) is a ($TM$-valued) Radon measure. We will say that  $V$ is \emph{$F$-stationary} if $\delta_F V \equiv 0$. In the isotropic setting $F_x(v)=|v|$, we will simply write $\delta V$ to denote the first variation of a varifold $V$.

\subsection{Anisotropic Allen--Cahn}
We now generalize the Allen--Cahn energy to the anisotropic setting by letting
$$E_\varepsilon(u) := \int_M \left[ \varepsilon \frac{F(x,\nabla u(x))^2}{2} + \frac{W(u(x))}{\varepsilon} \right] d\operatorname{vol}_g(x),$$
where $W$ is a fixed double-well potential vanishing at $\pm1$;
we assume that $W:\R\to[0,\infty)$ is smooth with
\begin{equation}\label{dw1}
W>0\text{ on }\R\setminus\{\pm1\},\quad W(\pm1)=0,\quad W''(\pm1)>0,
\quad (\sqrt{W})''\le-c<0\text{ on }(-1,1),
\end{equation}
as well as
\begin{equation}\label{dw2}
-Cs\le W'(s)<0\text{ on }(-\infty,-1),\quad 0<W'(s)\le Cs\text{ on }(1,\infty),
\quad cs^2\le W(s)\le Cs^2\text{ on }\R\setminus[-2,2],
\end{equation}
for two constants $c,C>0$.
Two standard choices are $W(s)=\frac{(1-s^2)^2}{4}$ or $W(s)=1+\cos(\pi s)$, suitably modified outside of the interval $[-1,1]$. The requirements \eqref{dw2} on $\R\setminus[-1,1]$ are actually
irrelevant in the construction of critical points $u$, since we can always modify $W$ on this set
in order to satisfy them, and the latter will imply that $|u|\le1$.

We will denote by 
\begin{equation}\label{eq:endens}
    e_{\e}(u(x)):=\varepsilon \frac{F(x,\nabla u(x))^2}{2} + \frac{W(u(x))}{\varepsilon}
\end{equation}
the energy density of $u:M\to \R$ with respect to $E_{\e}$.

We observe that, although $F^2$ is smooth away from the zero section of $TM$, it is in general only $C^1$ on $TM$. Indeed, if $F_x^2$ is $C^2$ at $0$, then by Taylor expansion there exists a symmetric matrix $Q$ such that
\[
F_x^2(v) = \langle Qv,  v \rangle + o(|v|^2).
\]
From the $2$-homogeneity of $F^2_x$ we obtain
\[
F_x^2(v) =\lim_{t \to 0} \frac{F_x^2(tv)}{t^2} = \langle Qv,  v \rangle.
\]
We deduce that $F^2_x$ is $C^2$ if and only if $F_x$ is the Euclidean norm up to linear changes of coordinates. Having a $C^1$ functional will be enough for first-order considerations and in particular to check the Palais--Smale condition. However, to give precise meaning to stability (or to Morse index bounds) of critical points, it will be useful to consider smooth approximations $F_\delta$ converging smoothly  to $F$ away from the zero section of $TM$, as $\delta \to 0$; see also Remark \ref{rem:smoothing} below.

To this aim, for each $\delta\in(0,1)$, we define for every $x\in M$
$$G_\delta(x,v):=
(F_x^2 * \eta_\delta)(v)-(F_x^2 * \eta_\delta)(0)\quad \mbox{and} \quad F_\delta(x,v):=\sqrt{G_\delta(x,v)}, $$
where $\eta_\delta$ is the standard radial mollifier supported in the ball $B_\delta(0)\subset T_xM\cong\R^n$
(using the metric $g$ for the latter identification). Note that $G_\delta(x,v)\ge0$
as it is convex and even in $v$, which forces $v=0$ to be the minimum point of $G_\delta(x,\cdot)$.
Hence,  $F_\delta^2$ is smooth, $F_\delta^2\to F^2$ in $C^1_{loc}(TM)$ as $\delta \to 0$,
and for all  $\delta\in(0,1)$
\begin{equation}\label{eq:l'}
    \mbox{$F_\delta$ satisfies \eqref{F.square.cvx}--\eqref{bound}, with a smaller $\lambda'>0$ uniform in $\delta$ in place of $\lambda$}.
\end{equation}
In fact, the validity of \eqref{bound} for $F_\delta$, namely $(\lambda'|v|)^2\le G_\delta(x,v)\le(|v|/\lambda')^2$, is immediate to check for fixed $\delta$
(if $|v|\ge 2$ it follows from \eqref{bound} for $F$, while near the origin it follows from
\eqref{F.square.cvx} and smoothness of $G_\delta(x,\cdot)$), and its uniformity follows from the fact that
$G_\delta(x,v)=\delta^2 G_1(x,v/\delta)$.
As an immediate consequence of \eqref{F.square.cvx}, the differential of $(F_\delta)_x^2$ is monotone, i.e.,
\begin{equation}\label{mono}
\langle D(F_\delta)_x^2(v)-D(F_\delta)_x^2(w),v-w\rangle \ge 2\lambda'|v-w|^2.
\end{equation}

We define the anisotropic Allen-Cahn energy associated with $F_\delta$ as follows:
$$E_{\epsilon,\delta}(u):=\int_M\left[\e \frac{F_\delta(x,\nabla u(x))^2}{2}+\frac{W(u(x))}{\e}\right]d\operatorname{vol}_g.$$
As in \eqref{eq:endens}, we denote by
$$e_{\e,\delta}(u):=\e \frac{F_\delta(x,\nabla u(x))^2}{2}+\frac{W(u(x))}{\e}$$ the energy density of $u:M\to \R$ with respect to $E_{\e,\delta}$.

\section{Existence of nontrivial solutions via min-max}\label{sec:Prel}
The main purpose of this section is to check the Palais--Smale condition for the functional $E_{\epsilon,\delta}$ and to deduce the existence of nontrivial critical points for $E_\e$.

\begin{proposition}\label{energyproperties}
The functional $E_{\epsilon,\delta}$ is finite and of class $C^2$ on $H^1(M)$, with
\begin{subequations}
	\begin{align*}
	DE_{\e,\delta}(u)[v] &=\int_M \left[\frac\e{2} D(F_\delta)_x^2(\nabla u)[\nabla v] + \frac1\e W'(u) v\right],  \\
	D^2E_{\e,\delta}(u)[v,w] &=\int_M \left[\frac\e{2} D^2(F_\delta)_x^2(\nabla u)[\nabla v,\nabla w] + \frac1\e W''(u) vw\right].
	\end{align*}
\end{subequations}
Moreover, it satisfies the Palais--Smale condition: 
if $u_k\in H^1(M)$ is a sequence such that $E_{\e,\delta}(u_k)$ is bounded uniformly in $k$ and $DE_{\e,\delta}(u_k)\to0$ strongly in $H^1(M)$ (note that $DE_{\e,\delta}(u_k)\in H^1(M)^*\cong H^1(M)$), then $(u_k)$ admits a strongly converging subsequence.
\end{proposition}

\begin{proof}
The finiteness of $E_{\epsilon,\delta}$ follows from the quadratic growth of $W$, while the $C^2$
regularity and the formulas for its derivatives are standard  calculations.
Since by \eqref{dw2} and \eqref{eq:l'}
$$(F_\delta)_x^2(\nabla u)\ge \lambda'|\nabla u|^2,\quad W(u)\ge c|u|^2-C,$$
given a sequence $(u_k)$ as in the statement, we immediately deduce that it is bounded in $H^1(M)$; by the Rellich--Kondrachov compactness theorem, there exists $u\in H^1(M)$ such that $u_k$ converges to $u$, weakly in $H^1(M)$ and strongly in $L^2(M)$, along a subsequence. 
We observe that
\begin{equation}\label{eq:conv1}
\begin{split}
    \lim_{k\to\infty} DE_{\e,\delta}(u)[u_k-u] &=  \lim_{k\to\infty} \int_M \left[\frac\e{2} D(F_\delta)_x^2(\nabla u)[\nabla u_k-\nabla u] + \frac1\e W'(u) (u_k-u)\right]=0,
\end{split}
\end{equation}
where for the first term on the right-hand side we use that $D(F_\delta)_x^2(\nabla u)\in L^2$
(as $|D(F_\delta)_x^2(v)|\le C|v|$) and $\nabla u_k-\nabla u\rightharpoonup 0$ in $L^2$, while for the second term we use that $|W'(u)|\le C|u|+C\in L^2$ by \eqref{dw2} and $u_k\to u$ strongly in $L^2$.

Since $DE_{\e,\delta}(u_k)\to0$ strongly in $H^1$ and $u_k$ is bounded in $H^1$ uniformly in $k$, we deduce that
$$\lim_{k\to\infty}DE_{\e,\delta}(u_k)[u_k-u]=0$$
and hence
\begin{equation}\label{utile}
    DE_{\e,\delta}(u_k)[u_k-u] - DE_{\e,\delta}(u)[u_k-u]\to0.
\end{equation}
Arguing again as for \eqref{eq:conv1}, we have $W'(u_k) (u_k-u)\to0$ in $L^1$, hence we deduce that
$$\int_M [D(F_\delta)_x^2(\nabla u_k)[\nabla u_k-\nabla u] - D(F_\delta)_x^2(\nabla u)[\nabla u_k-\nabla u]]\to0.$$
Using \eqref{mono}, this implies that
$$\int_M |\nabla u_k-\nabla u|^2 \to 0,$$
and hence $u_k\to u$ strongly in $H^1$,
as desired.
\end{proof}

\begin{remark}\label{rem:smoothing}
The original functional $E_\e$ is of class $C^1$ and the previous proof still
applies to show that $E_\e$ satisfies the Palais--Smale condition.
On the other hand, $E_\e$ cannot be used to give an immediate meaning to stability or to a bound on the Morse index
of critical points.
\end{remark}

\begin{proposition}\label{regularity}
	Any critical point $u$ of $E_{\e,\delta}$ is smooth and satisfies $|u|\le1$.
	Also, there exists $\alpha\in(0,1)$ depending only on $F$ such that, for every family of critical points $(u_{\e,\delta})_{\delta\in(0,1)}$ for the perturbed functionals $E_{\e,\delta}$ with $\sup_\delta E_{\e,\delta}(u_{\e,\delta})<\infty$, along a subsequence $u_{\e,\delta}$ converges in $C^{1,\alpha}(M)$ as $\delta \to 0$ to 
	a critical point
	$$u_\e=\lim_{\delta\to0}u_{\e,\delta} \in C^{1,\alpha}(M)$$
	for $E_\e$, with $E_\e(u_\e)=\lim_{\delta\to0}E_{\e,\delta}(u_{\e,\delta})$.
	Moreover, $u_\e$ is smooth on the open set $\{|\nabla u_\e|\neq0\}$.
\end{proposition}

\begin{proof}
    For each $u=u_{\e,\delta}$ the Euler--Lagrange equation for the functional $E_{\e,\delta}$ reads
    \begin{equation}\label{el.div}\operatorname{div}(a(\nabla u))=\e^{-2}W'(u),\end{equation}
    where in local coordinates $a(\nabla u)=g^{ij}\de_i((F_\delta)_x^2)(\nabla u)\de_j$. Note that the vector field
    $a$ is Lipschitz and, by \eqref{mono}, $a$ is also monotone, in the sense that
    $$\ang{a(v)-a(w),v-w}\ge \lambda'|v-w|^2, \quad \mbox{for all }v,w \in T_xM.$$
    The bound $u\le1$ readily follows by testing the integral form of \eqref{el.div} with $\varphi:=(u-1)^+$ and using the strict inequality $0<W'$ on $(1,\infty)$ in \eqref{dw2}. The bound $u\ge-1$ is obtained analogously. The fact that $u$ is smooth follows
    by expanding \eqref{el.div} into a linear second-order elliptic PDE, using the fact
    that $F_\delta^2$ is smooth.
    
    We can apply the results of \cite[Chapter 4]{LU} to the PDE \eqref{el.div} in divergence form to get the uniform bound
    $$\|u_{\e,\delta}\|_{C^{1,\alpha}}\le C,$$
    where $C$ may depend on $\e$ and the total energy $E_{\e,\delta}(u_{\e,\delta})$, but not on $\delta$, as $\sup_\delta E_{\e,\delta}(u_{\e,\delta})<\infty$.
    By Arzelà--Ascoli, we deduce the existence of a subsequential limit $u_\e$ in $C^{1,\alpha}$, up to slightly decreasing $\alpha$.
    The limit $u_\e$ is thus a critical point of $E_\e$. The smoothness of $u_\e$ on $\{\nabla u_\e\neq0\}$ follows by standard Schauder theory,  using that $F$ is smooth away from the zero section of $TM$.
\end{proof}

\begin{remark}\label{reg.bis}
In fact, any critical point $u_\e$ of $E_\e$ is $C^{1,\alpha}$, as one can see as follows:
for a fixed small ball $B$, we consider a minimizer $u_{\e,\delta}$ of $E_{\e,\delta}$ on $B$,
    with trace equal to $u_\e|_{\de B}$,  and deduce that $\|u_{\e,\delta}\|_{C^{1,\alpha}(B')}\le C$ on a smaller concentric ball $B'$. As $\delta\to0$, any subsequential limit in $H^1(B)$ must coincide with $u_\e$, as shown by the uniqueness result in \cite[Section 4]{LU}, showing that $u_\e$ is $C^{1,\alpha}$ on $B'$, and hence on $M$.
\end{remark}

%
%
%
We are now ready to construct 
nontrivial critical points for $E_\e$, satisfying uniform upper and lower bounds on the energy,
together with a stability property.
By \eqref{eq:l'} and \eqref{bound}, we deduce that the isotropic Allen--Cahn energy
$$\tilde E_\e(u):=\int_M\left[\e\frac{|\nabla u|^2}{2}+\frac{W(u)}{\e}\right]$$
satisfies
\begin{equation}\label{eq:equivalent}
    (\lambda')^2 \tilde E_\e(u)\le E_{\e,\delta}(u)\le(\lambda')^{-2}\tilde E_\e(u), \quad \mbox{for all }u\in H^1(M).
\end{equation} 
Letting
$$\Gamma:=\{\gamma:[-1,1]\to H^1(M)\,:\,\gamma\text{ continuous},\ \gamma(-1)\equiv-1,\ \gamma(1)\equiv1\}$$
and defining the mountain-pass values
$$c_{\e,\delta}:=\inf_{\gamma\in\Gamma}\max_{t\in[-1,1]}E_{\e,\delta}(\gamma(t)), \quad \tilde c_{\e,\delta}:=\inf_{\gamma\in\Gamma}\max_{t\in[-1,1]}\tilde E_{\e}(\gamma(t))$$
as in \cite{Guaraco2018}, we deduce from \cite[Proposition 5.2]{Guaraco2018} that
$$0<\liminf_{\e\to 0}\tilde c_\e\le\limsup_{\e\to 0}\tilde c_\e<\infty.$$
Combining this with \eqref{eq:equivalent}, we conclude that there exist $0<\beta\le\beta'$ and a small $\e_0>0$ such that
$$\beta\le c_{\e,\delta}\le\beta'\quad\text{for all }\delta\in(0,1),\ \e\in(0,\e_0).$$
We can now deduce the following.

\begin{proposition}\label{prop:bounds}
For $\e\in(0,\e_0)$ there exists a critical point $u_\e$ for $E_\e$,
with
$$E_\e(u_\e)\in[\beta,\beta'].$$
Moreover, $u_\e$ is a limit in $C^{1,\alpha}(M)$ of critical points $u_{\e,\delta}$
for $E_{\e,\delta}$ with Morse index $\le1$, along a sequence $\delta\to0$ depending on $\e$.
\end{proposition}

\begin{proof}
The existence of critical points $u_{\e,\delta}$ with energy in $[\beta,\beta']$ and Morse index $\le1$
follows from classical min-max theory on Banach spaces (see, e.g., \cite[Chapter 10]{Ghoussoub}).
The convergence along a subsequence to a critical point for $E_\e$ now follows from Proposition \ref{regularity}.
\end{proof}

The main contribution of the present work is to show that,
in such a situation (and even for more general min-max problems),
the energy density of $u_\e$ concentrates along an integral $F$-stationary varifold.

\section{$\Gamma$-convergence of $E_\e$ to $\mathcal{F}$}

As a preliminary result, in fact not needed in the next sections,
we prove in this section that $E_\epsilon$ $\Gamma$-converges to $c_W\cdot\mathcal{F}$,
where $c_W:=\int_{-1}^1\sqrt{2W}$ is a positive constant depending only on $W$. This was first proved by Bouchitté \cite{Bouchitte1990} in a more general framework. We provide a shorter proof in Theorem~\ref{thm:GammaConv}, tailored to $E_\e$. In the following we denote by $1_{S}$ the indicator function of a set $S$, i.e., $1_{S}=1$ on $S$ and $1_{S}=0$ on $M\setminus S$.

\begin{theorem}\label{thm:GammaConv}
    Given a sequence $\e_k\to0$ and maps $u_k\in H^1(M)$
    with
    $\liminf_{k\to\infty}E_{\epsilon_k}(u_k)<\infty$,
    there exists a finite perimeter set $S\subseteq M$ such that
    $$u_k\to 1_{S}-1_{M\setminus S}$$
    pointwise a.e.\ and in $L^2(M)$
    up to a subsequence, as well as
    $$c_W\cdot \mathcal{F}(S)\le\liminf_{k\to\infty}E_{\epsilon_k}(u_k).$$
    Conversely, given a finite perimeter set $S\subseteq M$,
    there exists a family of smooth maps $(u_\epsilon)_{\epsilon>0}$
    with $|u_\epsilon|\le1$,
    $$u_\epsilon\to 1_{S}-1_{M\setminus S}$$
    pointwise a.e.\ and in $L^p(M)$ for all $p<\infty$, and
    $$E_\epsilon(u_\epsilon)\to c_W\cdot \mathcal{F}(S).$$
\end{theorem}

\begin{proof}
    Given a sequence $(\e_k,u_k)$ as in the statement,
    we will write $\epsilon$ and $u_\epsilon$ in place of $\epsilon_k$ and $u_k$, with a slight abuse of notation.
    Up to a subsequence, we can assume that the $\liminf$ is a limit. Also,
    letting $v_\e:=\min\{\max\{u_\e,-2\},2\}$, note that by \eqref{dw2} we have
    $$E_\e(v_\e)\le E_\e(u_\e),\quad\int_M|u_\e-v_\e|^2\le\int_{\{|u_\e|>2\}}|u_\e|^2\le C\int_M W(u_\e)\le C\e.$$
    Thus, we can assume without loss of generality that $|u_\e|\le2$, up to replacing $u_\e$ with $v_\e$.
    
    Since $|\nabla u_\epsilon|\le \lambda^{-1}F_x(\nabla u_\epsilon)$, by Cauchy--Schwarz we have
    $$\int_M \sqrt{2W(u_\epsilon)}|\nabla u_\epsilon|\le \lambda^{-1}\int_M\sqrt{2W(u_\epsilon)}F_x(\nabla u_\epsilon)
    \le \lambda^{-1}E_\e(u_\e)\le C$$
    for some constant $C>0$ independent of $\epsilon$.
    Thus, letting $H(t):=\int_0^t \sqrt{2W(s)}\,ds$ and $w_\epsilon:=H(u_\epsilon)$,
    we have a uniform $BV$ bound:
    $$\int_M(|w_\epsilon|+|\nabla w_\epsilon|)\le C.$$
    By the compact embedding $BV(M)\hookrightarrow L^1(M)$,
    up to a subsequence we can find $w_0\in BV(M)$
    such that $w_\epsilon\to w_0$ in $L^1$ and pointwise a.e., and thus also in $L^2(M)$, as $|w_\e|\le2$.

    Since $H:\R\to\R$ is continuous and bijective and $w_\epsilon\to w_0$ pointwise a.e.,
    we have
    $$u_\epsilon=H^{-1}(w_\e)\to H^{-1}(w_0)=:u_0 \mbox{ pointwise a.e.},$$
    and the limit $u_0$ takes values in $\pm1$ a.e.\ since
    $$\int_M W(u_0)\le\liminf_{\epsilon\to0}W(u_\epsilon)=\lim_{\epsilon\to0}O(\epsilon)=0.$$
    Moreover, $S:=\{u_0=1\}=\{w_0=H(1)\}$ is a set of finite perimeter,
    as its indicator function is precisely $\frac{w_0-H(-1)}{H(1)-H(-1)}\in BV(M)$.
    The desired bound on the $\mathcal{F}$-perimeter follows easily from the convexity of $F$ at each $x\in M$: indeed, a straightforward adaptation
    of \cite[Theorem 2.38]{AFP} to the Riemannian case gives
    \begin{align*}
    \int_M F_x\left(\frac{\nabla w_0}{|\nabla w_0|}\right)\,d|\nabla w_0|
    &\le\liminf_{\epsilon\to0}\int_M F_x(\nabla w_\epsilon)\\
    &=\liminf_{\epsilon\to0}\int_M H'(u_\epsilon)F_x(\nabla u_\epsilon)\\
    &=\liminf_{\epsilon\to0}\int_M \sqrt{2W(u_\epsilon)}F_x(\nabla u_\epsilon)\\
    &\le\liminf_{\epsilon\to0}E_\epsilon(u_\epsilon).
    \end{align*}
    This gives the desired conclusion, since
    $$\nabla 1_S=\nabla\frac{w_0-H(-1)}{H(1)-H(-1)}=\frac{\nabla w_0}{H(1)-H(-1)}=\frac{\nabla w_0}{c_W}$$
    and hence
    $$\mathcal{F}(S)=\int_{\de^*S}F_x(\nu_x)\,d\mathcal{H}^{n-1}(x)
    =\frac{1}{c_W}\int_M F_x\left(\frac{\nabla w_0}{|\nabla w_0|}\right)\,d|\nabla w_0|.$$

    Conversely, given a set $S\subseteq M$ of finite perimeter, by \cite[Theorem 3.42]{AFP}
    we can find a sequence of smooth open sets $S_k$ such that
    $$\mathcal{H}^n(S_k\Delta S)\to0,\quad \operatorname{Per}(S_k)\to\operatorname{Per}(S),$$
    where $\operatorname{Per}$ denotes the isotropic perimeter. Thus, by Reshetnyak's continuity
    principle \cite[Theorem 2.39]{AFP}, we have the convergence of $\mathcal{F}$-perimeters:
    $$\int_{\de S_k}F_x(\nu_x)\,d\mathcal{H}^{n-1}(x)\to \int_{\de^*S}F_x(\nu_x)\,d\mathcal{H}^{n-1}(x).$$
    Thus, to prove the existence of a recovery sequence, we can just consider the case where $S$ is a smooth open set;
    the conclusion then follows by a standard diagonal argument. In this case, we consider the
    one-dimensional
    heteroclinic solution $U:\R\to\R$ such that
    $$U'=\sqrt{2W(U)},\quad U(0)=0,\quad\lim_{t\to\pm\infty}U(t)=\pm1$$
    and we let
    $$U_\gamma(t):=\begin{cases}
    -1&\text{for }t\le-2\gamma\\
    U(\gamma t/(2\gamma+t))&\text{for }t\in(-2\gamma,-\gamma]\\
    U(t)&\text{for }t\in[-\gamma,\gamma]\\
    U(\gamma t/(2\gamma-t))&\text{for }t\in[\gamma,2\gamma)\\
    1&\text{for }t\ge2\gamma.
    \end{cases} $$
    Using the exponential decay of $U'$ and of $W(U)$ at infinity, it is easy to check that, as $\gamma\to\infty$, we have
    $$\int_{\R\setminus[-\gamma,\gamma]}[(U')^2/2+W(U)]=O(e^{-c\gamma}),\quad
    \int_{-\gamma}^\gamma[(U')^2/2+W(U)]=\int_{-\gamma}^\gamma \sqrt{2W(U)}U'\to c_W.$$
    
    Then, taking $\delta>0$ small and writing $B_{\delta}(\de S)$ as a disjoint union of geodesics of the form
        $$\{\ell_p(t):=\exp_p(t\nu_p)\mid t\in(-\delta,\delta)\}$$
        as $p$ ranges in $\de S$, we can define
    $$u_{\e,\gamma}(\ell_p(t)):=U_\gamma\left(\frac{t}{\e F_p(\nu_p)}\right).$$
    As long as $2\gamma\cdot\e\max_{p\in \de S} F_p(\nu_p)<\delta$,
    we can extend this to a smooth map $u_{\e,\gamma}:M\to[-1,1]$ by $u_{\e,\gamma}:=-1$ on $S\setminus B_{\delta}(\de S)$
    and $u_{\e,\gamma}:=1$ on $(M\setminus S)\setminus B_{\delta}(\de S)$.
    We observe that
    $$\nabla u_{\e,\gamma}(\ell_p(t))=U_\gamma'\left(\frac{t}{\e F_p(\nu_p)}\right)\frac{\ell_p'(t)}{\e F_p(\nu_p)}+O(|t|/\e)=U_\gamma'\left(\frac{t}{\e F_p(\nu_p)}\right)\frac{\ell_p'(t)}{\e F_p(\nu_p)}+O(\gamma),$$
    where the error term comes from differentiation of $F\circ\nu$ at the nearest-point projection $p$;
    note that $|U_\gamma'|\le C$ and that both sides vanish when $|t|>2\gamma\e F_p(\nu_p)$.
    Moreover,
    $$|\nabla u_{\e,\gamma}|(\ell_p(t))\le C\e^{-1}e^{-c\gamma}\quad\text{for }\gamma\e F_p(\nu_p)<|t|<\delta$$
    and similarly
    $$W(u_{\e,\gamma})(\ell_p(t))\le Ce^{-c\gamma}\quad\text{for }\gamma\e F_p(\nu_p)<|t|<\delta.$$
    Thus, we have
    $$\int_{\{\gamma\e F_p(\nu_p)<|t|<\delta\}}[\e F_{\ell_p(t)}(\nabla u_{\e,\gamma})^2/2+\e^{-1}W(u_{\e,\gamma})](\ell_p(t))\,dt
    =O(e^{-c\gamma}),$$
    for a possibly different $c>0$, as the integrand is nonzero only on an interval of size $O(\gamma\e)$.
    Hence, 
    $$E_\epsilon(u_{\e,\gamma})=(1+O(\delta))\int_{\de S}\int_{\{|t|<\gamma\e F_p(\nu_p)\}}[\e F_{\ell_p(t)}(\nabla u_{\e,\gamma})^2/2+\e^{-1}W(u_{\e,\gamma})](\ell_p(t))\,dt\,d\mathcal{H}^{n-1}(p)+O(e^{-c\gamma}).$$
    Since
    $$F_{\ell_p(t)}(\nabla u_{\e,\gamma}(\ell_p(t)))=(1+O(\delta))\e^{-1}U_\gamma'\left(\frac{t}{\e F_p(\nu_p)}\right)+O(\gamma)$$
    by the previous expansion, we deduce that
    \begin{align*}
    E_\epsilon(u_{\e,\gamma})&=(1+O(\delta))\int_{\de S}\int_{\{|t|<\gamma\e F_p(\nu_p)\}}
    \left[\frac{1}{2\e}U_\gamma'\left(\frac{t}{\e F_p(\nu_p)}\right)^2+\frac{1}{\e}W(U_\gamma)\left(\frac{t}{\e F_p(\nu_p)}\right)\right]\,dt\,d\mathcal{H}^{n-1}(p)\\
    &\quad+O(\gamma^2\e)+O(e^{-c\gamma}),
    \end{align*}
    as long as $\gamma\e<1$. By a simple change of variables, we get
    $$E_\epsilon(u_{\e,\gamma})=(1+O(\delta))\int_{\de S}F_p(\nu_p)\,d\mathcal{H}^{n-1}(p)
    \cdot \int_{-\gamma}^\gamma[(U_\gamma')^2/2+W(U_\gamma)]+O(\gamma^2\e)+O(e^{-c\gamma}).$$
    This converges to $\mathcal{F}(S)$ as we let $\e\to0$, then $\gamma\to\infty$
    and $\delta\to0$.
    By a diagonal argument, the conclusion follows.
\end{proof}

\section{A generalization of Modica's bound}
The following pointwise bound for critical points $u$ of the isotropic Allen--Cahn
$$\e\frac{|\nabla u|^2}{2}\le\frac{W(u)}{\e}$$
was first proved by Modica \cite{Modica} in the Euclidean setting. This is the fundamental tool used in deriving a sharp monotonicity formula
in the isotropic case \cite{Ilmanen,HutchinsonTonegawa2000}.

In the anisotropic setting, we obtain an analogous bound; as expected, it no longer yields a monotonicity formula. Nonetheless, it will be a crucial ingredient
in the proof of rectifiability of the limit of the energy densities. 

From now on up to the end of the paper, given a map $u:M\to \R$, when evaluating $F(x,\nabla u(x))$ with a slight abuse of notation we will look at $\nabla u$ as a vector field in the tangent bundle $\nabla u: M\to TM$, and use the compact expression $F(\nabla u)$ in place of $F(\cdot,\nabla u(\cdot))$.

\begin{theorem}\label{thm:modicabound}
	Letting $u:M\to[-1,1]$ be a critical point of $E_\e$, we have
    $$F(\nabla u)\le\e^{-1}\sqrt{2W(u)}+ C,$$
    where $C$ depends only on $M^n,g$ and $F$.  As a consequence,
    $$\frac{\e F^2(\nabla u)}{2}\le\frac{W(u)}{\e}+C$$
    for a possibly different $C=C(M^n,g,F)$.
\end{theorem}

\begin{proof}
Recall that, by Proposition \ref{regularity} and Remark \ref{reg.bis}, away from $\{\nabla u=0\}$, $u$ is smooth and we can expand the Euler--Lagrange equation \eqref{el.div} as
$$A^{ij}\partial_{ij}u=\e^{-2}W'(u)+O(\nabla u),$$
where the coefficients $A^{ij}(x)$ are essentially the second derivatives of $F_x^2/2$ at $\nabla u$;
more precisely, in any coordinate chart we have
$$A^{ij}=g^{i\ell}g^{jm}\de_{\ell m}(F^2/2)(\nabla u)=g^{i\ell}g^{jm}[F(\nabla u)\partial_{\ell m}F(\nabla u)+\partial_\ell F(\nabla u)\partial_mF(\nabla u)],$$
where when differentiating $F$ and $F^2$ we use the convention that $\de_k F$ denotes the
partial derivative along $\de_k=\frac{\de}{\de x^k}\in T_xM$ of $F_x=F|_{T_xM}$, at any given $x\in M$
(and similarly for higher-order derivatives and for $F^2$).
The error term is a function $G(x,\nabla u)$, with $G(x,v)$ smooth away from $\{v=0\}$ and $1$-homogeneous in $v$,
so that $|\frac{\de^{\alpha+\beta}G}{\de x^\alpha\de v^\beta}|\le C_{\alpha,\beta}|v|^{1-|\beta|}$.

We would like to show that
$$F(\nabla u)-\e^{-1}\sqrt{2W(u)}\le\Lambda$$
everywhere, where $\Lambda>0$ will be chosen later. To do this, consider a maximum point $\hat x$ for the difference
$$\zeta:=F(\nabla u)-\e^{-1}\sqrt{2W(u)},$$
and assume by contradiction that $\zeta(\hat x)>\Lambda$,
so that in particular $\nabla u(\hat x)\neq0$, and thus $u$ is smooth around $\hat x$, as well as $|u(\hat x)|<1$.
We now choose a coordinate system centered at $\hat x$, with $g_{ij}(0)=\delta_{ij}$ and $\de_kg_{ij}(0)=0$, as well as
$$\nabla u(0)=|\nabla u(0)|e_n.$$
Since $A^{ij}$ is positive definite, we have
$$A^{ij}\partial_{ij}[F(\nabla u)-\e\sqrt{2W(u)}](\hat x)=A^{ij}\de_{ij}\zeta(\hat x)\le0.$$

We now compute that at $\hat x=0$ we have
\begin{align*}
\partial_{ij}[F(\nabla u)-\e^{-1}\sqrt{2W(u)}]
&=\partial_i[g^{k\ell}\partial_kF(\nabla u)\partial_{j\ell}u-\e^{-1}(\sqrt{2W})'(u)\de_ju+O(\nabla u)]\\
&=\partial_{k\ell}F(\nabla u)\partial_{i\ell}u\partial_{jk}u
+\partial_kF(\nabla u)\partial_{ijk}u-\e^{-1}(\sqrt{2W})'(u)\de_{ij}u\\
&\quad-\e^{-1}(\sqrt{2W})''(u)\de_iu\de_ju+O(\nabla u)+O(D^2 u),
\end{align*}
where the error terms come from differentiating $F$ in the spatial variable at least once.

Once we multiply by $A^{ij}$ and sum over $i,j$, we get
\begin{align}\label{eq.with.I}
    I
    +A^{ij}\partial_kF(\nabla u)\partial_{ijk}u
    -\e^{-3}\frac{W'(u)^2}{\sqrt{2W(u)}}
    -\e^{-1}(\sqrt{2W})''(u)F^2(\nabla u)
    &\le O(\e^{-1}\nabla u)+O(D^2 u),
\end{align}
where we omit the sum over $i,j$ and we set
$$I:=A^{ij}\partial_{k\ell}F(\nabla u)\partial_{i\ell}u\partial_{jk}u$$
and we used the Euler--Lagrange equation and the fact that
$$A^{ij}\partial_iu\partial_ju=D^2(F^2/2)(\nabla u)[\nabla u,\nabla u]=F^2(\nabla u),$$
since $F^2$ is 2-homogeneous.

We rewrite the second term of \eqref{eq.with.I} as
\begin{align*}
    &\partial_k(A^{ij}\partial_{ij}u)\partial_kF(\nabla u)-\partial_kA^{ij}\partial_kF(\nabla u)\partial_{ij}u\\
    &=\e^{-2}W''(u)\partial_k u\partial_kF(\nabla u)
    -\partial_{ij\ell}(F^2/2)(\nabla u)\partial_{k\ell}u\partial_kF(\nabla u)\partial_{ij}(u)+O(\nabla u)+O(D^2u),
\end{align*}
where we used again the Euler--Lagrange equation and
we expanded $\partial_kA^{ij}$ using the chain rule.
Since $F$ is 1-homogeneous, at $\hat x$ we have
$$\partial_k u\partial_kF(\nabla u)=F(\nabla u).$$
Moreover, since $\nabla\zeta(\hat x)=0$, at $\hat x$ we have
$$\partial_{k\ell}u\partial_kF(\nabla u)=\partial_\ell[F(\nabla u)]+O(\nabla u)=\e^{-1}\partial_\ell\sqrt{2W(u)}+O(\nabla u)=\e^{-1}(\sqrt{2W})'(u)\de_\ell u
+O(\nabla u),$$
giving
$$\partial_{ij\ell}(F^2/2)(\nabla u)\partial_{k\ell}u\partial_kF(\nabla u)\partial_{ij}(u)
=\e^{-1}(\sqrt{2W})'(u)\partial_{ij\ell}(F^2/2)(\nabla u)\de_\ell u\partial_{ij}(u)
+O(D^2u),$$
thanks to the fact that $\partial_{ij\ell}(F^2/2)$ is $(-1)$-homogeneous, so that
$|\partial_{ij\ell}(F^2/2)(\nabla u)|\le C|\nabla u|^{-1}$.
Also, since $\de_{ij}(F^2/2)$ is $0$-homogeneous, we have
$$\partial_{ij\ell}(F^2/2)(\nabla u)\de_\ell u=0.$$
In summary, at $\hat x$ we get
$$A^{ij}\partial_kF(\nabla u)\partial_{ijk}u
=\e^{-2}W''(u) F(\nabla u)+O(D^2u).$$
Thus, \eqref{eq.with.I} becomes
$$I+\e^{-2}W''(u)F(\nabla u)
-\e^{-3}\frac{W'(u)^2}{\sqrt{2W(u)}}
    -\e^{-1}(\sqrt{2W})''(u)F^2(\nabla u)
    \le O(\e^{-1}\nabla u)+O(D^2 u).$$
    
Now $I$ is nonnegative, since it is the trace of the product of two positive semidefinite matrices; in fact,
by \eqref{F.cvx} and the fact that $D^2F(v)[v,\cdot]=0$ for $v\neq0$ (by $0$-homogeneity of $DF_x$),
we have
\begin{align*}
I&:=A^{ij}\partial_{k\ell}F(\nabla u)\partial_{i\ell}u\partial_{jk}u\\
&\phantom{:}=\operatorname{tr}((D^2F)(D^2uAD^2u))\\
&\phantom{:}\ge \frac{\lambda}{|\nabla u|}\operatorname{tr}((I-e_n^*\otimes e_n^*)(D^2uAD^2u))\\
&\phantom{:}\ge \frac{\lambda^2}{|\nabla u|}\operatorname{tr}((I-e_n^*\otimes e_n^*)(D^2u)^2)\\
&\phantom{:}\ge \frac{\lambda^2}{|\nabla u|}\sum_{i=1}^n\sum_{j=1}^{n-1}|\de_{ij}u|^2,
\end{align*}
where we repeatedly used the fact that $\operatorname{tr}(AB)\ge\operatorname{tr}(A'B')$ if $A\ge A'\ge0$ and $B\ge B'\ge0$.
Moreover, by the Euler--Lagrange equation, the lower bound $A^{nn}\ge\lambda$, and the estimate $|W'|\le C\sqrt{W}$ (as $W''(\pm1)>0$), we have
\begin{equation}\label{nn}|\de_{nn}u|\le C\sum_{(i,j)\neq(n,n)}|\de_{ij}u|+O(\nabla u)+C\e^{-2}\sqrt{W(u)}.
\end{equation}
By the assumption that $\e^{-1}\sqrt{W(u)}<F(\nabla u)$ at $\hat x$, the estimate \eqref{nn} in turn implies that at $\hat x$
$$|\de_{nn}u|\le C\sum_{(i,j)\neq(n,n)}|\de_{ij}u|+O(\e^{-1}\nabla u).$$
Thus, by Cauchy's inequality,
we can absorb the term $O(D^2u)$, getting
\begin{align*}
I+O(D^2u)\ge -C|\nabla u|-C\e^{-1}|\nabla u|\ge -C\e^{-1}F(\nabla u)
\end{align*}
at $\hat x$, where $C$ depends on the implied constant in $O(D^2u)$. We deduce that
\begin{equation}\label{eq:int}
    \e^{-2}W''(u)F(\nabla u)
-\e^{-3}\frac{W'(u)^2}{\sqrt{2W(u)}}
    -\e^{-1}(\sqrt{2W})''(u)F^2(\nabla u)
    \le C\e^{-1}F(\nabla u).
\end{equation}

Since $W''=(\sqrt{2W})''\cdot\sqrt{2W}+\frac{(W')^2}{2W}$,
we can rewrite \eqref{eq:int} as
\begin{align*}-\e^{-1}(\sqrt{2W})''(u)F(\nabla u)[F(\nabla u)-\e^{-1}\sqrt{2W(u)}]
+\e^{-2}\frac{(W')^2}{2W}[F(\nabla u)-\e^{-1}\sqrt{2W(u)}]
\le -I+O(\e^{-1}\nabla u)\end{align*}
at $\hat x$; recalling that $-(\sqrt{2W})''\ge c>0$ by \eqref{dw1},
we reach
$$c\e^{-1}F(\nabla u)[F(\nabla u)-\e^{-1}\sqrt{2W(u)}]+\e^{-2}\frac{(W')^2}{2W}[F(\nabla u)-\e^{-1}\sqrt{2W(u)}]\le C\e^{-1}F(\nabla u).$$
This contradicts the fact that $F(\nabla u)-\e^{-1}\sqrt{2W(u)}>\Lambda$, once we take $\Lambda$ large enough that $c\Lambda\ge C$.
\end{proof}

\section{Uniform bounds for stable solutions}

In this section we derive bounds on the second fundamental form of level sets for stable solutions $u_{\e,\delta}$ with respect to $E_{\e,\delta}$ and consequently for limits $u_\e$ thereof. Moreover we obtain lower density bounds for the energy. Since $F^2$ is not of class $C^2$, in many statements
we will replace it with the perturbed integrands $F_\delta^2$ already considered earlier, in order to make sense of stability.
Let us start with a simple observation.

\begin{proposition}\label{zero.unstable}
    There exists $\rho=\rho(M^n,g,F)>0$ large enough
    such that the following holds.
    Given $x_0\in M$, we identify $T_{x_0}M\cong\R^n$ isometrically, so that
    the restriction $(F_\delta)_{x_0}=F_\delta|_{T_{x_0}M}$ gives an autonomous integrand $\bar F_\delta:\R^n\to[0,\infty)$.
    If a constant $\bar u\in[-1,1]$ is a stable critical point for
    $$\bar E_{1,\delta}(u):=\int_{B_\rho(0)}[{\bar F_\delta(\nabla u)^2}/{2}+W(u)]$$
    on the Euclidean ball $B_\rho(0)\subset\R^n$,
    then $\bar u\in\{-1,1\}$.
\end{proposition}

\begin{proof}
Indeed, assume by contradiction that $u_0\in(-1,1)$. For any $\eta\in C^1_c(B_\rho(0))$, stability gives
$$\int_{B_\rho(0)}[D^2(\bar F_\delta^2/2)(0)[\nabla\eta,\nabla\eta]+W''(\bar u)\eta^2]\ge0.$$
Since $W'(\bar u)=W'(\pm1)=0$ and $W''(\pm1)>0$, we must have $|\bar u|\le 1-c$ for some $c>0$.
Also, using again the fact that $W'(\bar u)=0$, we have
$$W''(\bar u)=2(\sqrt{W})''(\bar u)\sqrt{W(\bar u)}.$$
Using \eqref{dw1}, we see that $-W''(\bar u)\ge c$ for a possibly different $c>0$.
We deduce that
\begin{equation}\label{const.stab}c\int_{B_\rho}\eta^2\le \int_{B_\rho}D^2(\bar F_\delta^2/2)(0)[\nabla\eta,\nabla\eta]
\le C\int_{B_\rho}|\nabla\eta|^2\end{equation}
for all $\eta\in C^1_c(B_\rho)$. This is impossible once we take $\rho>0$ large enough.
\end{proof}

The following are useful consequences.

\begin{lemma}\label{jac.lb}
    Given $\gamma>0$,
    there exist constants $c>0$ and $\epsilon_0>0$, depending on $\gamma$ and $(M^n,g,F)$, such that
    $$\int_{B_{\rho\epsilon}(p)}\sqrt{2W(u)}|\nabla u|\ge c\epsilon^{n-1}$$
    whenever $|u(p)|\le1-\gamma$,
    for any stable critical point $u:B_{\rho\e}(p)\to\R$ of $E_{\e,\delta}$, provided that
    $\epsilon\in(0,\epsilon_0)$.
\end{lemma}

\begin{proof}
    Arguing by contradiction, dilating the domain by a factor $\e^{-1}$, assume that
    $$\int_{B_{\rho}^{(\epsilon)}}\sqrt{2W(\tilde u_\epsilon)}|\nabla\tilde u_\epsilon|\to0$$
    for a sequence of rescaled functions $\tilde u_\epsilon$, defined on rescaled geodesic balls $B_\rho^{(\epsilon)}$
    and stably critical for $E_{1,\delta_\e}$, for some $\delta_\e\in(0,1)$.
    Since $\e\to0$, the rescaled metrics converge to the Euclidean one.
    
    Since $u_\e$ is uniformly $C^{1,\alpha}$ on $B_{\rho'}^{(\e)}$, for any $\rho'\in(0,\rho)$ (see Proposition \ref{regularity} and its proof), once we identify each $B_{\rho}^{(\epsilon)}$ with the Euclidean ball $B_\rho(0)$, these solutions 
    converge in $C^1_{loc}$, along a subsequence, to a critical point $\tilde u_0:B_\rho(0)\to[-1,1]$
    for $E_{1,\delta_0}$, where $\delta_0:=\lim_{\e\to0}\delta_\e\in[0,1]$. Here the limit energy $E_{1,\delta_0}$ involves the autonomous integrand $F_{p_0}$, where $p_0=\lim_{\e\to0}p_\e$.
    	Also, $|\tilde u_0(0)|=\lim_{\e\to0}|u_\e(p_\e)|\le 1-\gamma$ and
    $$\int_{B_{\rho}(0)}\sqrt{2W(\tilde u_0)}|\nabla\tilde u_0|=0,$$
    so that $\nabla\tilde u_0$ vanishes on the open set $\{\tilde u_0\in(-1,1)\}$.
    We deduce that $\tilde u_0$ is a constant value in $[-1+\gamma,1-\gamma]$.
    If $\delta_0>0$, then $\tilde u_0$ is also stable and we can immediately apply Proposition \ref{zero.unstable}
        to reach a contradiction.
    
    If $\delta_0=0$, we can still derive an inequality like \eqref{const.stab}: namely, for $\e$ small enough, depending on $\operatorname{spt}(\eta)$, we can write
    $$c\int_{B_\rho^{(\e)}}\eta^2\le C\int_{B_\rho^{(\e)}}|\nabla\eta|^2,$$
    by exploiting the stability of $\tilde u_\e$ and using the fact that $W''(\tilde u_\e)\to W''(\tilde u_0)\le-c<0$
    locally uniformly.
    In the limit we get $c\int_{B_\rho}\eta^2\le C\int_{B_\rho}|\nabla\eta|^2$,
    for all $\eta\in C^1_c(B_\rho)$, which is again impossible by our choice of $\rho$ in the proof of Proposition \ref{zero.unstable}, making \eqref{const.stab} fail.
\end{proof}

\begin{lemma}\label{f.lb}
    There exist constants $c>0$ and $\epsilon_0>0$, depending on $(M^n,g,F)$, such that
    $$\int_{B_{2r}(p)}\epsilon F_\delta(\nabla u)^2\ge c\int_{B_r(p)}\frac{W(u)}{\epsilon}$$
    for any stable critical point $u:B_{2r}(p)\to\R$ of $E_{\e,\delta}$, provided that
    $\rho\epsilon<\min\{r,\epsilon_0\}$.
\end{lemma}

\begin{proof}
    Arguing by contradiction, assume that a stable solution $u_\e$ satisfies the reverse inequality for arbitrarily small $c,\epsilon>0$. Since we can cover $B_r(p)$ by a family of balls $B_{\rho\epsilon}(x)\subseteq B_{2r}(p)$ with bounded overlap, we can find a sequence of stable solutions
    $u_{\epsilon}$ (with $\epsilon\to0$) and points $x_\epsilon\in M$ such that
    \begin{equation}\label{contrad1}
    \int_{B_{2\rho\epsilon}(x_\epsilon)}\epsilon F(\nabla u_\epsilon)^2<c_\epsilon\int_{B_{\rho\epsilon}(x_\epsilon)}\frac{W(u_\epsilon)}{\epsilon},
    \end{equation}
    with $c_\epsilon\to0$.

    We can rescale $u_\epsilon$ to functions $\tilde u_\epsilon$, defined on rescaled geodesic balls $B_{2\rho}^{(\epsilon)}$ converging to the Euclidean ball $B_{2\rho}(0)$.
    As in the previous proof, we have $\tilde u_\e\to\tilde u_0$ in $C^1_{loc}$ up to a subsequence.
    Since
    $$\int_{B_{2\rho}^{(\epsilon)}}|\nabla\tilde u_\epsilon|^2\le C\int_{B_{2\rho}^{(\epsilon)}}F(\nabla \tilde u_\epsilon)^2<C\cdot c_\epsilon\int_{B_\rho^{(\epsilon)}}W(\tilde u_\epsilon)\to0,$$
    in the limit we deduce that $\tilde u_0$ is constant
    on $B_{2\rho}(0)$.
    
    If $u_\e$ is critical and stable for $E_{\e,\delta_\e}$ with $\delta_0:=\lim_{\e\to0}\delta_\e>0$,
    then $\tilde u_0$ is critical and stable for $E_{1,\delta_0}$, so that Proposition
    \ref{zero.unstable} gives $\tilde u_0\in\{\pm1\}$; if instead $\delta_0=0$,
    we can reach the same conclusion by arguing exactly as in the previous proof.
    
    Now, recalling that $W''(\pm1)>0$, we can fix $\gamma\in(0,1)$ such that
    $$-\operatorname{sgn}(s)W'(s)(1-|s|)\ge 4c(1-|s|)^2\ge(1-s^2)^2\text{ for }|s|\in[1-\gamma,1],$$
    while, for a possibly different $c>0$, we also have
    $$c(1-s^2)^2\le W(s)\le C(1-s^2)^2\quad\text{for }|s|\le1.$$
    Since $|\tilde u_0|=1$, we have $|\tilde u_\epsilon|\in[1-\gamma,1]$ on $B_{3\rho/2}^{(\epsilon)}$ eventually.

    Let us take a cut-off function $\phi_\epsilon\in C^1_c(B_{3\rho/2}^{(\epsilon)})$, equal to $1$ on $B_{\rho}^{(\epsilon)}$ and with $|\nabla\phi_\e|\le4\rho^{-1}$.
    Assuming for instance that $\tilde u_\epsilon\ge 1-\gamma$ on the ball $B_{3\rho/2}^{(\epsilon)}$ and
    testing the Euler--Lagrange equation \eqref{el.div} with $\phi_\epsilon^2(1-\tilde u_\epsilon)$, we find
    $$-\int_{B_{2\rho}^{(\epsilon)}}\phi_\epsilon^2 W'(\tilde u_\epsilon)(1-\tilde u_\epsilon)=\int_{B_{2\rho}^{(\epsilon)}}\langle a_\e(\nabla\tilde u_\e), \nabla(\phi_\epsilon^2(1-\tilde u_\epsilon))\rangle,$$
    for suitable $a_\e:TM\to TM$ with $|a_\e(v)|\le C|v|$.
    Applying Young's inequality, we deduce that
    $$-\int_{B_{2\rho}^{(\epsilon)}}\phi_\epsilon^2 W'(\tilde u_\epsilon)(1-\tilde u_\epsilon)\le C(\sigma)\int_{B_{2\rho}^{(\epsilon)}}|\nabla \tilde u_\epsilon|^2+\sigma\int_{B_{2\rho}^{(\epsilon)}}\phi_\epsilon^2 (1-\tilde u_\epsilon)^2$$
    for an arbitrarily small $\sigma>0$. Using the fact that $\tilde u_\e\in[1-\gamma,1]$ on the support of $\varphi_\e$,
    we deduce that
    $$c\int_{B_{2\rho}^{(\epsilon)}}\phi_\epsilon^2 (1-\tilde u_\epsilon^2)^2
    \le C(\sigma)\int_{B_{2\rho}^{(\epsilon)}}|\nabla \tilde u_\epsilon|^2+\sigma\int_{B_{2\rho}^{(\epsilon)}}\phi_\epsilon^2 (1-\tilde u_\epsilon)^2,$$
    and thus, taking $\sigma:=c/2$, we get
    $$\int_{B_{2\rho}^{(\epsilon)}}\phi_\epsilon^2 (1-\tilde u_\epsilon^2)^2
        \le C\int_{B_{2\rho}^{(\epsilon)}}|\nabla \tilde u_\epsilon|^2,$$
        or alternatively
    $$\int_{B_{2\rho}^{(\epsilon)}}\phi_\epsilon^2 W(\tilde u_\epsilon)^2
            \le C\int_{B_{2\rho}^{(\epsilon)}}F(\nabla \tilde u_\epsilon)^2.$$
    Thus, we have
    $$\int_{B_{\rho}^{(\epsilon)}}W(\tilde u_\epsilon)^2
                \le C\int_{B_{2\rho}^{(\epsilon)}}F(\nabla \tilde u_\epsilon)^2,$$
    which contradicts \eqref{contrad1}. The case $\tilde u_\e\in[-1,-1+\gamma]$ is analogous.
\end{proof}

We record here a diffuse version of the \emph{stability inequality} for $\mathcal{F}$-stationary hypersurfaces,
similar to the one obtained for the isotropic Allen--Cahn \cite{PadillaTonegawa1998}, first for stable critical points
of $E_{\e,\delta}$.
We will then let $\delta\to0$ to derive a consequence for critical points $u_\e$ of $E_\e$ which are limits $u_\e=\lim_{\delta\to0}u_{\e,\delta}$
of stable critical points of $E_{\e,\delta}$.

\begin{theorem}\label{stab.ineq}
Assume that $u$ is a stable critical point for $E_{\e,\delta}$ on an open set $U\subseteq M$.
Then we have
$$\int_U\varphi^2|\operatorname{II}_{u}|^2\cdot \e F_\delta(\nabla u)^2\le C(\varphi)\int_U e_{\e,\delta}(u),$$
where $\operatorname{II}_{u}(x)$ denotes the second fundamental form of the level set $\{u=u(x)\}$ if $\nabla u(x)\neq 0$, and it is set to be zero on $\{\nabla u=0\}$.
\end{theorem}

\begin{proof}
Since $u_\delta$ is stable on $U$, for any $\phi\in C^2_c(U)$
a straightforward computation shows that
$$\int_U \left[A^{ij}_\delta\de_i\phi\de_j\phi+\frac{W''(u_\delta)}{\epsilon^2}\phi^2\right]\ge0,$$
where as above $A^{ij}_\delta=g^{i\ell}g^{jm}\de_{\ell m}(F_\delta^2/2)(\nabla u)$, and $\de_{\ell m}(F_\delta^2/2)$ denotes the second derivative of $(F_\delta^2/2)|_{T_xM}$ along $\de_\ell,\de_m\in T_xM$, at any given $x\in M$.

Since $|\nabla u|$ is Lipschitz, by a standard approximation argument we can plug $\phi|du_\delta|$ in place of $\phi$,
obtaining
\begin{align}\label{stab.one}\int_U \left[\phi^2\left(A^{ij}_\delta\de_i|\nabla u|\de_j|\nabla u|+\frac{W''(u)}{\epsilon^2}|\nabla u|^2\right)
+A^{ij}_\delta\de_i\phi\de_j\phi |\nabla u|^2
+2A^{ij}_\delta\phi|\nabla u|\de_i\phi\de_j|\nabla u|\right]\ge0.\end{align}
Next, testing criticality with $\phi^2\Delta u$, we have
$$\int_U \left[g^{ij}\de_i (F_\delta^2/2)(\nabla u)\de_j(\phi^2\Delta u)+\frac{W'(u)}{\epsilon^2}\phi^2\Delta u\right]=0.$$
Writing $\de_j(\phi^2\Delta u)=\phi^2 g^{k\ell}\de_{kj\ell}u
+\varphi[O(\nabla u)+O(D^2u)]$, with implied constants depending on $\varphi$, and noting that
$$\de_k[\de_i (F_\delta^2/2)(\nabla u)]
=g^{pq}\de_{ip}(F_\delta^2/2)(\nabla u)\de_{qk}u,$$
after an integration by parts and relabeling of indices we obtain
\begin{align*}
&\int_U \left[\phi^2 A^{ij}_\delta g^{k\ell}\de_{ik}u \de_{j\ell}u
+\phi^2\frac{W''(u)}{\epsilon^2}|\nabla u|^2
+\e^{-2}\ang{\nabla(\phi^2),\nabla(W(u))}\right]\\
&=\int_U \varphi [O(|\nabla u|^2)+O(\nabla u)O(D^2u)],
\end{align*}
and hence
\begin{align}\label{stab.two}
\int_U\left[ \phi^2 A^{ij}_\delta g^{k\ell}\de_{ik}u \de_{j\ell}u
+\phi^2\frac{W''(u)}{\epsilon^2}|\nabla u|^2\right]
&=\int_U [\e^{-1}\cdot O(e_\e(u))+\varphi\cdot O(\nabla u)O(D^2u)].
\end{align}


Subtracting \eqref{stab.one} from \eqref{stab.two}, we get
$$\int_U \phi^2 A^{ij}_\delta[g^{k\ell}\de_{ik}u \de_{j\ell}u-\de_i|\nabla u|\de_j|\nabla u|]\le C(\phi)\int_U [\e^{-1} e_\e(u)+\varphi|\nabla u||D^2u|].$$
We now absorb the last error term: fixing a point $p\in M$ where $\nabla u(p)\neq0$ and choosing a chart centered at $p$,
with $g_{ij}(p)=\delta_{ij}$, $\de_k g_{ij}(p)=0$, and $\nabla u(p)=|\nabla u(p)|\de_n$, as in \eqref{nn} we observe that
$$|\de_{nn}u(p)|\le C\sum_{(i,j)\neq(n,n)}|\de_{ij}u|,$$
while the integrand in the left-hand side is equal to $\phi^2$ times
$$A^{ij}_\delta[\de_{ik}u \de_{jk}u-\de_{in}u\de_{jn}u]\ge\lambda'\sum_{i=1}^n\sum_{k=1}^{n-1}|\de_{ik}u|^2.$$
On the other hand, the same term equals $A^{ij}_\delta\de_{ik}u \de_{jk}u\ge\lambda'\sum_{i,k=1}^n|\de_{ik}u|^2$
at a.e.\ point where $|\nabla u|=0$.
Thus, by Cauchy's inequality, we reach the bound
\begin{equation}\label{bdsub}\int \phi^2 A^{ij}_\delta[g^{k\ell}\de_{ik}u \de_{j\ell}u-\de_i|\nabla u|\de_j|\nabla u|]\le C(\phi)\int_U \e^{-1} e_\e(u),\end{equation}
for a possibly larger $C(\phi)$.

Let $A_\delta$ denote the section of $TM\otimes TM$ with components $A^{ij}_{\delta}$ (in a coordinate chart),
which (using the metric $g$) we can recast as a positive definite $(1,1)$-tensor $A_\delta\ge\lambda'I$.
On $\{\nabla u\neq0\}$, letting $\nu:=\frac{\nabla u}{|\nabla u|}$, we observe that the integrand on the left-hand side can be written more compactly as
$$\phi^2 \operatorname{tr}((HA_\delta H)(g-\nu^*\otimes\nu^*))\ge\lambda'\phi^2\operatorname{tr}(H^2(g-\nu^*\otimes\nu^*)),\quad H:=D^2u.$$
Given a point $p$ and selecting a coordinate chart as above, writing $e_i:=\de_i$, for $i=1,\dots,n-1$ we have
$$|\operatorname{II}_u(e_i)|^2=\sum_{j=1}^{n-1}|\ang{\de_i\nu,e_j}|^2=|\de_i\nu|^2$$
at $p$, as the tangent space of the level set is spanned by $\{e_1,\dots,e_{n-1}\}$, while $e_n=\nu$. Moreover
$$\de_i\nu=\frac{H(e_i)-H(e_i,\nu)\nu}{|\nabla u|},$$
so that for all $j=1,\dots,n-1$ we have
$$\ang{\de_i\nu,e_j}=\frac{H(e_i,e_j)}{|\nabla u|}.$$
Thus,
$$|\operatorname{II}_u|^2=\sum_{i,j=1}^{n-1}\left|\frac{H(e_i,e_j)}{|\nabla u|}\right|^2
\le\frac{\operatorname{tr}(H^2(g-\nu^*\otimes\nu^*))}{|\nabla u|^2}.$$
We deduce that
$$\int_U\phi^2|\operatorname{II}_u|^2|\nabla u|^2\le C\phi^2 \operatorname{tr}((HA_\delta H)(g-\nu^*\otimes\nu^*)).$$
The claim now follows from \eqref{bdsub}.
\end{proof}


Now, given a critical point $u_\e$ for the energy $E_\e$, we define the $(n-1)$-dimensional varifold
$\tilde V_\e$ to be the measure on the Grassmannian bundle $G:=G_{n-1}(M)$ given by
\begin{equation}\label{def:average}
    \tilde V_\epsilon(f):=\int_{\{\nabla u_\e\neq0\}}\sqrt{2W(u_\e)}|\nabla u_\e|\cdot f(P_{u_\e})\,d\operatorname{vol}_g,
\end{equation}
for any $f\in C^0(G)$, where $P_{u_\e}(x):=(\nabla u_\e(x))^\perp\in G$ is the tangent plane to the level set $\{u_\e=u_\e(x)\}$
at $x$. Here we make a slight abuse of notation: we write $f(P_{u_\e}(x))$ in place of $f(x,P_{u_\e}(x))$.

\begin{remark}\label{avg.lev}
If almost all level sets $\{u_\e=\lambda\}$ are regular then, viewing them as the $(n-1)$-varifolds $\llbracket\{u_\e=\lambda\}\rrbracket$,
we have
$$\tilde V_\e=\int_{-1}^1\sqrt{2W(\lambda)}\cdot \llbracket\{u_\e=\lambda\}\rrbracket\,d\lambda,$$
by the coarea formula. In other words, in this case $\tilde V_\e$ is simply a weighted average of the level sets of $u_\e$.
\end{remark}

Further, given a sequence of critical points $(u_\e)$ with $\e\to0$ and
\begin{equation}\label{e.bded}
\liminf_{\e\to 0}E_\e(u_\e)<\infty,
\end{equation}
we can assume that the $\liminf$ is a limit (up to a subsequence) and define the measure
$$ d\mu=\lim_{\e\to0}e_\e(u_\e)\,d\operatorname{vol}_g,$$
namely $\mu$ is the limit of the energy densities in duality with $C^0(M)$, up to a subsequence.

\begin{corollary}\label{firstvar.bd.cor}
	Assuming \eqref{e.bded}, up to a subsequence we can extract a limit varifold $\tilde V_0$, with weight
	$$\|\tilde V_0\|\le C\mu.$$
    If moreover each $u_\e=\lim_{\delta\to0}u_{\e,\delta}$, for a suitable sequence
    $(u_{\e,\delta})_{\delta}$ of critical points for $E_{\e,\delta}$  with Morse index $\le m$ independently of $\e,\delta$,
    then there exists a finite set of points $\mathcal{S}$, with $\#\mathcal{S}\le m$, such that any
    $p\not\in\mathcal{S}$ admits a neighborhood $U$ where the \emph{isotropic} first variation $\delta\tilde V_0$ satisfies
    $$|\delta\tilde V_0|(U)\le C(U).$$
\end{corollary}

The sequence $\delta=\delta_k\to0$ used in the limit $u_\e=\lim_{\delta\to0}u_{\e,\delta}$ is allowed to depend on $\e$.
In the last part of the statement, $C(U)$ denotes a finite constant which may depend on all data (in particular, on $U$).

\begin{proof}
    The total weight of $\tilde V_\epsilon$
    on an open set $U\subseteq M$ equals
    $$\|\tilde V_\epsilon\|=\int_U\sqrt{2W(u_\epsilon)}|\nabla u_\epsilon|
    \le C\int_U e_\epsilon(u_\epsilon),$$
    by Cauchy--Schwarz and the bound $|\nabla u_\e|^2\le\lambda^{-2}F(\nabla u_\e)^2$,
    showing the first claim.
    
    Now assume that $u_\e=\lim_{\delta\to0}u_{\e,\delta}$ and that each $u_{\e,\delta}$ is stable on $U$.
    Then, recalling that the convergence is in $C^{1,\alpha}(M)$ and that the integrand defining $\tilde V_\e$ contains the weight $|\nabla u|$, which vanishes on the complement
    of $\{\nabla u\neq0\}$, it is straightforward to check that in the sense of varifolds
    $$\tilde V_\e=\lim_{\delta\to0}\tilde V_{\e,\delta}.$$
     Now, since $u_{\e,\delta}$ is smooth, we can apply Remark \ref{avg.lev}
    to say that $\tilde V_{\e,\delta}$ is a weighted average of its level sets.
    By subadditivity of the first variation, we have
    \begin{align*}
    |\delta\tilde V_{\e,\delta}|(U)
    &\le\int_{-1}^1\sqrt{2W(\lambda)}\cdot|\delta\llbracket \{u_{\e,\delta}=\lambda\}\rrbracket|(U)\,d\lambda\\
    &=\int_{-1}^1\sqrt{2W(\lambda)}\int_{\{u_{\e,\delta}=\lambda\}\cap U}|H_{u_{\e,\delta}}|\,d\mathcal{H}^{n-1}\,d\lambda\\
    &\le C\int_U|\operatorname{II}_{u_{\e,\delta}}|\sqrt{2W(u_{\e,\delta})}|\nabla u_{\e,\delta}|,
    \end{align*}
    where we denoted by $H_{u_{\e,\delta}}$ the mean curvature of the level set and used again the coarea formula.
    Since trivially
    $$\int_M \e^{-1}W(u_{\e,\delta})\le E_{\e,\delta}(u_{\e,\delta})\le C,$$
    from Theorem \ref{stab.ineq} and Cauchy--Schwarz we get
    $$|\delta\tilde V_{\e,\delta}|(U)\le C(U).$$
    By lower semicontinuity of the first variation, we deduce that
    \begin{equation}\label{first.var.ve}
    |\delta\tilde V_\e|(U)\le\liminf_{\delta\to 0}|\delta\tilde V_{\e,\delta}|(U)\le C(U),
    \end{equation}
    and thus
    $$|\delta\tilde V_{0}|(U)\le\liminf_{\e\to 0}|\delta\tilde V_{\e}|(U)\le C(U).$$
    
    The conclusion is standard: for every $\e,\delta$, let $S_{\e,\delta,r}$ be the set of points $p$
    such that $u_{\e,\delta}$ is unstable on $B_r(p)$. By Vitali's covering lemma,
    there exists a subcollection $S_{\e,\delta,r}'$ such that the balls $B_r(p)$ with $p\in S_{\e,\delta,r}'$
    are disjoint, giving in particular $\# S_{\e,\delta,r}'\le m$, and such that
    $u_{\e,\delta}$ is stable on $B_r(q)$ unless $q\in\bigcup_{p\in S_{\e,\delta,r}'}B_{5r}(p)$.
    The desired set is obtained by taking a limit of $S_{\e,\delta,r}'$ in the Hausdorff topology
    as $\delta\to0$, then as $\e\to0$, and finally as $r\to0$ (along subsequences).
\end{proof}

We are now in a position to prove the following. In fact, later on we will
just use the consequence \eqref{mu.dens},
even if the rectifiability of $\tilde V_0$ will simplify the argument. Later, $\tilde V_0$ will be replaced by
a more appropriate varifold $V$, directly tied to the anisotropic first variation; the latter will be shown to be ($c_W$ times) integer rectifiable on all of $M$.

\begin{theorem}\label{rect}
	On $M\setminus\mathcal{S}$ we have
	$$c\mu\le\|\tilde V_0\|\le C\mu$$
    and the density of $\tilde V_0$ satisfies
    \begin{equation}\label{tilde.dens}\Theta^{n-1}(\|\tilde V_0\|,x):=\lim_{r\to0}\frac{\|\tilde V_0\|(B_r(x))}{r^{n-1}}\in(0,\infty)
    \end{equation}
    at $\|\tilde V_0\|$-almost every $x\in M\setminus\mathcal{S}$ (or equivalently at $\mu$-a.e.\ $x\in M\setminus\mathcal{S}$),
    and hence $\tilde V_0$ is rectifiable on $M\setminus\mathcal{S}$.
    As a consequence, $\mu$ is a rectifiable measure on $M\setminus\mathcal{S}$ and
    \begin{equation}\label{mu.dens}
    0<\Theta^{n-1}_*(\mu,x)
    \le\Theta^{n-1,*}(\mu,x)<\infty\quad\text{for $\mu$-a.e.}\ x\in M\setminus\mathcal{S},
    \end{equation}
    where $\Theta^{n-1}_*(\mu,x):=\liminf_{r\to0}\frac{\mu(B_r(x))}{r^{n-1}}$
    and $\Theta^{n-1,*}(\mu,x):=\limsup_{r\to0}\frac{\mu(B_r(x))}{r^{n-1}}$.
\end{theorem}

\begin{proof}
	While the upper bound $\|\tilde V_0\|\le C\mu$ was already obtained (on all of $M$),
	we claim that the lower bound $\|\tilde V_0\|\ge c\mu$ on $M\setminus\mathcal{S}$ follows from \eqref{tilde.dens}.
	Indeed, given $p\not\in\mathcal{S}$, let us fix a ball $B_{2r}(p)\subseteq U$,
	where $U$ is the neighborhood given by the previous result.
	First of all, by Theorem \ref{thm:modicabound}, we have
	$$\|\tilde V_\epsilon\|(B_{2r}(p))\ge c\int_{B_{2r}(p)}\sqrt{2W(u_\e)}F(\nabla u_\e)
	\ge \int_{B_{2r}(p)}[c\e F(\nabla u)^2-C\e F(\nabla u)]
	\ge c\int_{B_{2r}(p)}\e F(\nabla u)^2-Cr^n,$$
	thanks to the uniform $C^1$ bound $F(\nabla u_\e)\le C\e^{-1}$, which again follows from Theorem \ref{thm:modicabound}.
	
	By Lemma \ref{f.lb}, which applies since $u_{\e,\delta}$ is stable on $U$ by construction (for $\e$, and in turn $\delta$, small enough), we deduce that
	$$\|\tilde V_\epsilon\|(B_{2r}(p))\ge c\int_{B_r(p)}e_\e(u_\e)-Cr^n,$$
	and hence in the limit $\e\to0$ we get $\|\tilde V_0\|(\bar B_{2r}(p))\ge c\mu(B_r(p))-Cr^n$. Approximating $r$ from below, we deduce 
	$$\|\tilde V_0\|(B_{2r}(p))\ge c\mu(B_r(p))-Cr^n.$$
	 Assuming that $\|\tilde V_0\|(B_r(p))\ge c(p)r^{n-1}$
	for any $r>0$ small enough, we can then find $r>0$ as small as we want and such that
	$$\|\tilde V_0\|(B_{2r}(p))\ge c(p)r^{n-1},\quad \|\tilde V_0\|(B_{r}(p))\ge 4^{-n}\|\tilde V_0\|(B_{2r}(p)).$$
	We then have $r^n=o(\|\tilde V_0\|(B_r(p)))$, and thus
	$$\|\tilde V_0\|(B_{r}(p))\ge 4^{-n}\|\tilde V_0\|(B_{2r}(p))\ge c\mu(B_r(p)),$$
	from which the bound $\|\tilde V_0\|\ge c\mu$ follows, thanks to Besicovitch's differentiation theorem.
	
	To show the second part of the statement, it suffices to show that $\Theta^{n-1,*}(\|\tilde V_0\|,x)>0$ for $\|\tilde V_0\|$-a.e.\ $x$:
	indeed, we already proved in Corollary \ref{firstvar.bd.cor} that the varifold $\tilde V_0$ has locally bounded (isotropic) first variation in $M\setminus\mathcal{S}$, so that here its rectifiability follows by Allard's classical rectifiability criterion \cite[Section 5]{allard1972first}, which gives \eqref{tilde.dens}.
	Finally, \eqref{mu.dens} is now clear on $M\setminus\mathcal{S}$.

	To check that $\Theta^{n-1,*}(\|\tilde V_0\|,x)>0$ for $\|\tilde V_0\|$-a.e.\ $x$, let us fix $U$ as above, and fix $U'\subset U$ and $r_0>0$ such that
    $B_{2r_0}(x)\subseteq U$ for all $x\in U'$.
    Given $\gamma>0$, we let $\mathcal{G}_{\epsilon,\gamma}$ denote the set of points $x\in U'$ such that
    \begin{equation}\label{g.gamma}
    |u_\epsilon(x)|\le1-2\gamma,\quad|\delta\tilde V_\epsilon|(B_r(x))\le\frac{\|\tilde V_\epsilon\|(B_{r}(x))}{\gamma}
    \end{equation}
    for all $r\in(0,r_0)$. Since $|\delta\tilde V_\epsilon|(U)\le C(U)$ by \eqref{first.var.ve}, a simple application of Besicovitch's covering lemma gives
    $$\|\tilde V_\epsilon\|(U'\setminus\mathcal{G}_{\epsilon,\gamma})
    \le C(U)\gamma+\int_{U'\cap\{|u_\epsilon|>1-2\gamma\}}\sqrt{2W(u_\epsilon)}|\nabla u_\epsilon|.$$
    We claim that the last term vanishes in the limit $\e,\gamma\to0$. Indeed,
    testing the Euler--Lagrange equation \eqref{el.div} with $\lambda-u_\e$ we have
    $$\int_M \e^{-1}W'(u_\e)(\lambda-u_\e)\le C\int_M \e|\nabla u_\e|^2\le CE_\e(u_\e)\le C;$$
    choosing $\lambda\in(-1,1)$ to be the maximum point of $W|_{[-1,1]}$
    (recall that $\sqrt{W}$ is assumed to be strictly concave on $[-1,1]$),
    we have $(\lambda-s)W'(s)\ge0$ for all $s\in[-1,1]$, and actually
    $$(\lambda-s)W'(s)\ge -c\operatorname{sgn}(s)W'(s)\ge c\sqrt{W(s)}\quad\text{for }|s|\in[1-2\gamma,1]$$
    for $\gamma>0$ small enough, as $W''(\pm1)>0$. We infer that
    $$\int_{\{|u_\e|>1-2\gamma\}}\e^{-1}\sqrt{W(u_\e)}\le C,$$
    and thus, since $\sqrt{W(s)}\le C(1-|s|)$ on $[-1,1]$ (again as $W''(\pm1)>0$),
    \begin{equation}\label{W.bd.gamma}
        \int_{\{|u_\e|>1-2\gamma\}}\e^{-1}{W(u_\e)}\le C\gamma.
    \end{equation}
    By Theorem \ref{thm:modicabound} and the last two bounds, we have shown that
    $$\int_{\{|u_\e|>1-2\gamma\}}\sqrt{2W(u_\epsilon)}|\nabla u_\epsilon|
    \le C\int_{\{|u_\e|>1-2\gamma\}}[\e^{-1}W(u_\epsilon)+\sqrt{W(u_\e)}]
    \le C\gamma+C\e.$$
    In summary, we deduce that there exists a small $\gamma_0>0$ such that
    $$\|\tilde V_\epsilon\|(U'\setminus\mathcal{G}_{\epsilon,\gamma})
    \le C(U)(\gamma+\e)$$
    for all $\gamma\in(0,\gamma_0)$ and $\e>0$ small.

    Up to a subsequence, let $\mathcal{G}_{0,\gamma}$
    denote the Hausdorff limit of the closures
    $\bar{\mathcal{G}}_{\epsilon,\gamma}$, along a subsequence $\e\to0$ depending on $\gamma\in(0,\gamma_0)$. We claim that
    $$\Theta^{n-1}(\|\tilde V_0\|,x)>0$$
    at each $x\in \mathcal{G}_{0,\gamma}$.
    Indeed, thanks to \eqref{g.gamma}, we can apply \cite[Theorem 17.6]{Simon} to the varifold $\tilde V_\epsilon$
    with $\alpha=1$; note that here we use geodesic balls rather than Euclidean ones, but the proof still carries through. Hence we obtain
    $$\frac{\|\tilde V_\epsilon\|(B_r(x))}{r^{n-1}}
    \ge c(\gamma)\frac{\|\tilde V_\epsilon\|(B_{\rho\epsilon}(x))}{(\rho\epsilon)^{n-1}}\ge c(\gamma)>0$$
    for any $x\in\mathcal{G}_{\epsilon,\gamma}$ and $r\in(\rho\epsilon,r_0)$,
    thanks to Lemma \ref{jac.lb}, which applies as $B_{\rho\e}(x)\subseteq U$ and $|u_{\e,\delta}(x)|\le1-\gamma$ for $\delta$ small enough, by \eqref{g.gamma}. 
    Thus, we also have
    $$\frac{\|\tilde V_0\|(\bar B_r(x))}{r^{n-1}}
    \ge c(\gamma)>0$$
    for any $x\in\mathcal{G}_{0,\gamma}$. Finally, we have
    $$\|\tilde V_0\|(U'\setminus\mathcal{G}_{0,\gamma})
    \le\liminf_{\epsilon\to0}
    \|\tilde V_\epsilon\|(U'\setminus\bar{\mathcal{G}}_{\epsilon,\gamma})\le C(U)\gamma.$$
    Since $\|\tilde V_0\|$ has positive density
    at all points in $\bigcup_{\gamma\in(0,\gamma_0)}\mathcal{G}_{0,\gamma}$ and we can cover $M\setminus\mathcal{S}$
    with countably many such $U'$, the statement follows. Note that the subsequence $\e\to0$ defining $\mathcal{G}_{0,\gamma}$ depends on $\gamma$,
    but this is irrelevant.
\end{proof}



\section{Stress-energy tensor and integrality of the limit varifold}
Given $u_{\varepsilon}:M\to\R$ critical for $E_\e$, a straightforward computation using inner variations
shows that, for any vector field $X\in C^1(M,TM)$, denoting by $DX$ the $(1,1)$-tensor given by the Levi-Civita
connection, we have
\begin{equation}\label{div.free}\int_M\ang{T_\e,DX}=\int_M O(\e|\nabla u|^2|X|),\end{equation}
where $T_\e$ is the stress-energy tensor, namely the $(1,1)$-tensor given by
\begin{equation}\label{strtens}
    T_\e:=e_\epsilon(u_\e)I-\epsilon \nabla u_\e\otimes D(F^2/2)(\nabla u_\e),
\end{equation}
with $D(F^2/2)(\nabla u_\e)\in T_x^*M$ denoting the differential of $F_x^2/2=(F^2/2)|_{T_xM}$
(viewed as a function $T_xM\to\R$) at $\nabla u(x)$, for any given $x\in M$; the error term comes from the fact that $F_x^2$ depends on $x$.

In the sequel, it is useful to consider the map
$$C_F(\nu):=I-\frac{\nu\otimes DF(\nu)}{F(\nu)},$$
where $\nu\in TM$ is a unit vector.
We observe that
$$T_\e=[e_\e(u_\e)-\e F(\nabla u_\e)^2]I+\e F(\nabla u_\e)^2 C_F(\nu_\e),\quad\nu_\e:=\frac{\nabla u_\e}{|\nabla u_\e|} \quad \mbox{on } \{\nabla u_\e\neq 0\},$$
while we let $\nu_\e:=0$ and $C_F(\nu_\e):=0$ on $\{\nabla u_\e=0\}$.
We have the trivial bound
$$|T_\e|\le Ce_\e(u_\e).$$
Assuming \eqref{e.bded}, let $T_0$ be a limit of the measures $T_\e\,d\operatorname{vol}_g$, up to a subsequence.

\begin{proposition}\label{prop:kernel}
The limit $T_0$ has the form
$$dT_0=((1-\lambda)I+\lambda A)\,d\mu+d\xi,$$
where $\lambda:M\to[0,1]$ and the measure $\xi$ satisfies $|\xi|\le C\operatorname{vol}_g$,
while for $\mu$-a.e.\ $x$ the tensor $A(x)$ belongs to the (compact) convex hull
$$\mathcal{C}_x:=\operatorname{co}(\{C_F(\nu)\,:\,\nu\in T_xM,\ |\nu|=1\}).$$
\end{proposition}

\begin{proof}
Indeed, we have
$$T_\e=[(1-\lambda_\e)I+\lambda_\e C_F(\nu_\e)] e_\e(u_\e)+\xi_\e,$$
where
\begin{equation}\label{lambda.e}
\lambda_\e:=\min\left\{\frac{\e F(\nabla u_\e)^2}{e_\e(u_\e)},1\right\}\in[0,1]
\end{equation}
and defined to be zero on $\{e_\e(u_\e)=0\}$, and
$$\xi_\e=-(\e F(\nabla u_\e)^2-e_\e(u_\e))^+I+(\e F(\nabla u_\e)^2-e_\e(u_\e))^+C_F(\nu_\e).$$
Crucially, by the generalization of Modica's bound, namely Theorem \ref{thm:modicabound},
we have $\e F(\nabla u_\e)^2-e_\e(u_\e)\le C$ pointwise, and thus
$$|\xi_\e|\le C.$$
Since $(1-\lambda_\e)I+\lambda_\e C_F(\nu_\e)\in\operatorname{co}(\{I\}\cup\mathcal{C}_x)$
at any given $x\in M$, it follows that
$$dT_0=B\,d\mu+d\xi,$$
where $\xi$ is the limit of $\xi_\e\,d\operatorname{vol}_g$ and
$B(x)\in\operatorname{co}(\{I\}\cup\mathcal{C}_x)$ a.e., as desired.
\end{proof}

\begin{theorem}\label{varifold.constr}
	Assuming that $u_\e=\lim_{\delta\to 0} u_{\e,\delta}$ is a limit of critical points for $E_{\e,\delta}$
	with Morse index $\le m$, then the measures $T_0$ and $\mu$ are rectifiable. In fact, we have
	$$\lambda=0,\quad A(x)=C_F(\nu_x)$$
	for a suitable unit vector $\nu_x\in T_xM$, at $\mu$-a.e.\ $x\in M$.
    Furthermore, writing $d\mu=\theta\,d(\mathcal{H}^{n-1}\res\Sigma)$ for a suitable rectifiable Borel set $\Sigma\subset M$ with $\sigma$-finite $\mathcal{H}^{n-1}$ measure and $\theta:\Sigma\to(0,\infty)$, the $(n-1)$-dimensional varifold
    $$dV(x,\nu^\perp):=\frac{\theta(x)}{F(\nu_x)} \delta_{\nu_x}(\nu)\otimes d(\mathcal{H}^{n-1}\res\Sigma)(x)$$
    is rectifiable and $F$-stationary, and $\nu_x$ is $\mathcal{H}^{n-1}$-a.e.\ the unit normal to $\Sigma$
    (unique up to sign).
\end{theorem}

    It follows from the formula for $V$ that its anisotropic energy is
    $$\FF(V)=\mu(M),\quad\FF(V;U)=\mu(U)\text{ for all }U\subseteq M\text{ Borel}.$$

\begin{proof}
We observe that $\mu$-a.e.\ $x\in M$ is an approximate continuity point of $\lambda$ and $A$,
and moreover by Theorem \ref{rect} satisfies
\begin{equation}\label{mu.dens.bis}
0<\Theta^{n-1}_*(\mu,x)\le\Theta^{n-1,*}(\mu,x)<\infty,\end{equation}
provided that $x\not\in\mathcal{S}$.
In this case, thanks to Proposition \ref{prop:kernel}, any blow-up $\tilde T_0$ of $T_0$ at any such point $x_0\in M$ will be of the form
$$d\tilde T_0=B\,d\tilde\mu,$$
where
\begin{equation}\label{B.form}B=(1-\lambda_0)I+\lambda_0\int_{\mathbb{S}^{n-1}}C_{F_{x_0}}(\nu)\,d\alpha(\nu)\end{equation}
is a constant matrix given by a suitable $\lambda_0\in[0,1]$ and a probability measure $\alpha$ on $\mathbb{S}^{n-1}$, and $\tilde \mu$ is a blow-up of $\mu$ at $x_0$.
Indeed, $\xi$ disappears in the blow-up, as $|\xi|(B_r(x_0))=o(r^{n-1})=o(\mu(B_r(x_0)))$, thanks to the fact that $\mu(B_r(x_0))\ge cr^{n-1}$ for $r$ small.
Note that a blow-up defined on $\R^n$ exists also when $x_0\in\mathcal{S}$, in which case $\tilde\mu=\delta_0$.

As guaranteed by \eqref{mu.dens.bis} (or by $\tilde\mu=\delta_0$), the measure $\tilde\mu$ cannot be a constant multiple of the Lebesgue measure
on $T_{x_0}M\cong\R^n$, as $\Theta^{n-1}_*(\tilde\mu,0)>0$. On the other hand, \eqref{div.free} implies
$$|\ang{T_0,DX}|\le C\int_M|X|\,d\mu;$$
given $Y\in C^1_c(\R^n,\R^n)$, we can plug $X(p):=\sum_{i=1}^nY^i(r^{-1}\exp_{x_0}^{-1}(p))e_i(p)$, extended to zero outside the domain of $\exp_{x_0}^{-1}$, for a fixed orthonormal frame $\{e_i\}_{i=1}^n$
defined near $x_0$ inducing the chosen identification $T_{x_0}M\cong\R^n$. Letting $r\to0$, we deduce that
$$\ang{\tilde T_0,DY}=0\quad\text{for all }Y\in C^1_c(\R^n,\R^n).$$

Taking $Y$ of the form $Y=\phi v$, we then see that $\tilde\mu$ is invariant along $B^\top v$.
Since $\tilde\mu$ is not a multiple of the Lebesgue measure, we deduce that
$$\operatorname{ker}(B)\neq\{0\}.$$
Let then $\nu_0\in\operatorname{ker}(B)$ be a unit vector.
Recalling \eqref{B.form}, we see that
\begin{equation}\label{eq:exp}
    0=\ang{DF_{x_0}(\nu_0),B\nu_0}=(1-\lambda_0)F_{x_0}(\nu_0)+\lambda_0\int_{\mathbb{S}^{n-1}}\ang{DF_{x_0}(\nu_0),C_{F_{x_0}}(\nu)\nu_0}\,d\alpha(\nu).
\end{equation}
By $1$-homogeneity and convexity of $F_{x_0}$, we have that
$$DF_{x_0}(\nu)[\nu]=F(\nu),\quad DF_{x_0}(\nu)[\nu']=F_{x_0}(\nu)+DF_{x_0}(\nu)[\nu'-\nu]\le F_{x_0}(\nu')$$
and, since $F$ is even, it also holds that $|DF_{x_0}(\nu)[\nu_0]|\le F_{x_0}(\nu_0)$ and $|DF_{x_0}(\nu_0)[\nu]|\le F_{x_0}(\nu)$. Hence, we obtain that
\begin{equation}\label{atomic}\ang{DF_{x_0}(\nu_0),C_{F_{x_0}}(\nu)\nu_0}=F_{x_0}(\nu_0)-\frac{DF_{x_0}(\nu)[\nu_0]}{F(\nu)}DF_{x_0}(\nu_0)[\nu]\ge0.\end{equation}
Plugging \eqref{atomic} in \eqref{eq:exp}, we deduce that $\lambda_0=1$. Moreover, since $F_{x_0}$ is strictly convex along non-radial directions, \eqref{atomic} can be an equality only
when $\nu=\pm\nu_0$.

Thus, $\alpha$ is concentrated on $\{\pm\nu_0\}$ so that, using again that $F$ is even, we obtain
$$B=A(x_0)=I-\frac{\nu_0\otimes DF_{x_0}(\nu_0)}{F(x_0)}.$$
For the same reason, $\operatorname{ker}(B(x_0))=\operatorname{span}\{\nu_0\}$
and thus the image of $B^\top $ is $\nu_0^\perp$. Hence, $\tilde\mu$ is invariant
along the hyperplane $\nu_0^\perp$.

We cannot have $x_0\in\mathcal{S}$, since otherwise we would have
$\tilde\mu=\delta_0$, a contradiction with the invariance of $\tilde \mu$ along $\nu_0^\perp$. Hence, $\mu(\mathcal{S})=0$
and the conclusion follows from the rectifiability of $\mu$ proved in Theorem \ref{rect},
since for generic $x_0$ any blow-up must be a constant multiple of $\mathcal{H}^{n-1}\res\nu_\Sigma(x_0)^\perp$, yielding
$\nu_0=\pm\nu_\Sigma(x_0)$.
The $F$-stationarity of $V$ is equivalent to the fact that $T_0$ is divergence-free.
\end{proof}

\begin{remark}
The previous proof generalizes the \emph{atomic condition} found in \cite[Definition 1.1]{de2018rectifiability}, which in codimension one characterizes convex integrands $F$ that are strictly convex along non-radial directions: see \cite[Theorem 1.3]{de2018rectifiability}. In fact, we could avoid appealing
to the rectifiability of $\|\tilde V_0\|$ and rely just on \eqref{mu.dens}: once we reach $\lambda_0=0$,
the rectifiability of $\mu$ follows from \cite[Lemma 2.2]{de2018rectifiability}.
\end{remark}

\begin{remark}\label{rmk:xizero}
    We observe that
    \begin{equation}\label{limit.disc}
        \lim_{\e\to0}\int_M(\e F(\nabla u_\e)^2-e_\e(u_\e))^+=\lim_{\e\to0}\int_M\left(\frac\e{2} F(\nabla u_\e)^2-\frac1\e W(u_\e)\right)^+=0.
    \end{equation}
    Indeed, any weak limit $\upsilon$ of the integrand (as a measure) satisfies $\upsilon\le C\mu$, as well as $\upsilon\le C\operatorname{vol}_g$ by Theorem \ref{thm:modicabound}. Since the measures $\mu$ and $\operatorname{vol}_g$  are mutually singular, we deduce that $\upsilon=0$ and hence \eqref{limit.disc}.
    In particular, we also have
    $$\xi=0.$$
\end{remark}

\begin{remark}\label{vfd.conv}
    Recalling the definition of $\lambda_\e$ from \eqref{lambda.e}, on the compact set
    $$\hat G:=\{\operatorname{co}(\{I\}\cup\mathcal{C}_x)\mid x\in M\}$$
    we can consider the positive measures
    $$dV_\e(x,Z):=\delta_{B_\e(x)}(Z)e_\e(u_\e)(x)\,d\operatorname{vol}_g(x),$$
    where
    $$B_\e(x):=(1-\lambda_\e(x))I+\lambda_\e (x)C_F(\nu_\e(x))=\frac{T_\e(x)-\xi_\e(x)}{e_\e(u_\e)(x)}$$
    (the last equality holding on $\{e_\e(u_\e)>0\}$),
    and a subsequential limit $V_0$.
    Letting $\pi:\hat G\to M$ denote the canonical projection
    and considering the disintegration of $dV_0(x,Z)=\alpha_x(Z)\otimes d\mu(x)$ with respect to $\pi$, let $V_0'$ denote the $\hat G$-valued measure on $M$
    given by replacing each probability measure $\alpha_x$ with its center of mass, namely
    $$dV_0'(x):=Z_x\, d\mu(x),\quad Z_x:=\int_{\pi^{-1}(x)}Z\,d\alpha_x(Z).$$
    It is straightforward to check that, if we have a sequence of measures $W_k\rightharpoonup W_\infty$ in $\hat G$,
    then $W_k'\rightharpoonup W_\infty'$. In particular, since $V_\e\rightharpoonup V_0$ and $V_\e'=(T_\e-\xi_\e)\,d\operatorname{vol}_g\rightharpoonup T_0$ as $\xi_\e\rightharpoonup 0$ by Remark \ref{rmk:xizero}, we deduce that $V_0'=T_0$.
    Exactly the same argument used in the previous proof then shows that the disintegration $\alpha_x$ of $V_0$ consists of a Dirac mass at $\mu$-a.e.\ $x\in M$, i.e., $V_0=V$ up to identifying $C_F(\nu)$ with $\nu^\perp$.
\end{remark}

We are now left to prove that $c_W^{-1}V$ is actually an integral varifold, or equivalently
$$\frac{\theta(x)}{c_W F(\nu_x)}\in\mathbb N, \quad \mbox{for }\|V\|-a.e.\ x,$$
where we recall that $c_W=\int_{-1}^1 \sqrt{2W(t)}\,dt$.

We can fix a $\mu$-generic point $x_0$,
so that we can assume $x_0\not\in\mathcal{S}$ as $\mu(\mathcal S)=0$, \eqref{mu.dens} holds at $x_0$, and $\mu$
has an approximate tangent plane at $x_0$;
as seen in the previous proof, we can also assume that $\mu$ and
$T_0$ blow up to
$$\theta\cdot\mathcal{H}^{n-1}\res\nu^\perp,\quad \theta C_{F_{x_0}}(\nu)\cdot\mathcal{H}^{n-1}\res\nu^\perp,$$
respectively, where $\nu=\nu_{x_0}$ and $\theta=\theta(x_0)$.
Further, by \eqref{limit.disc} and \eqref{W.bd.gamma} we have
$$\limsup_{\e\to0}\int_{\{|u_\e|>1-2\gamma\}}e_\e(u_\e)
\le \limsup_{\e\to0}\int_{\{|u_\e|>1-2\gamma\}}2W(u_\e)
+\limsup_{\e\to0}\int_M\left(\frac\e{2} F(\nabla u_\e)^2-\frac1\e W(u_\e)\right)^+\le C\gamma.$$
Hence, an application of Besicovitch's covering lemma shows that, for a constant $C$ which may depend on the point $x_0$, we can also assume that
\begin{equation}\label{bd.en.gamma}
\limsup_{\e\to0}\int_{B_r(x_0)\cap\{|u_\e|\ge1-\gamma\}}e_\e(u_\e)\le C\gamma\mu(\bar B_r(x_0))\le C\gamma r^{n-1}\quad \text{for all }r\in(0,1).\end{equation}

We choose coordinates so that $x_0=0$ and the weak tangent plane is ($\theta$ times)
$$P:=\{ x \in \mathbb{R}^n \,:\, x^n = 0 \}.$$
By a diagonal argument, we can replace
$(u_\e)$ with a sequence of critical points, defined on
larger and larger balls $B_{R_\e}(0)$ endowed with
metrics $g_\e$ converging to the Euclidean one and integrands $F^\e\to F_{x_0}$,
such that
$$e_\e(u_\e)\,dx\rightharpoonup\theta\,d\mathcal{H}^{n-1}\res P,$$
as well as
$$T_\e\,dx\rightharpoonup \theta C_{F_{x_0}}(e_n)\,d\mathcal{H}^{n-1}\res P.$$

\begin{remark}
    Since the new sequence $u_\e$ is obtained by rescaling a sequence of critical points on a closed manifold,
    the pointwise bound \eqref{eq:modicabound} still holds. Actually, we have
    \begin{equation}\label{modica.blowup}
        \e F(\nabla u_\e)^2/2\le \e^{-1}W(u_\e)+c_\e,\quad c_\e\to0,
    \end{equation}
    since if the new $u_{\e}$ is obtained by rescaling a function $u_{\e'}$ from the original sequence
    then the constant $C$ in \eqref{eq:modicabound} becomes $C\e'/\e=C/\rho$, where $\rho$ is the (larger and larger) dilation factor.
\end{remark}

Moreover, using \eqref{bd.en.gamma} and another diagonal argument, we can even assume that
\begin{equation}\label{bd.en.gamma.bis}\limsup_{\e\to0}\int_{B_{10}(0)\cap\{|u_\e|\ge1-\gamma\}}e_\e(u_\e)\le C\gamma\end{equation}
for any $\gamma\in\{2^{-k}\mid k\in\N\}$, and hence for all $\gamma\in(0,1)$.


Let $\eta>0$ and $S_\epsilon:=\{|\nu_\e-e_n|>\eta\}$, where $\nu_\e=\frac{\nabla u_\epsilon}{|\nabla u_\epsilon|}$ is defined using the metric $g_\e$,
while $|\nu_\e-e_n|$ is computed using the Euclidean metric.
Then Remark \ref{vfd.conv}, namely the fact that
$$\delta_{B_\e(x)}(Z)e_\e(u_\e)(x)\,d\operatorname{vol}_g(x)=dV_\e(x,Z)\rightharpoonup
dV_0(x,Z)=\delta_{C_{F_{x_0}}(e_n)}(Z)\,d(\theta\mathcal{H}^{n-1}\res P)(x),$$
 gives
$$\int_{B_{10}(0)\cap S_\epsilon}e_\e(u_\e)
=V_\e(\pi^{-1}(B_{10}(0)\cap S_\e))\to0.$$
On the other hand, on $B_{10}(0)\setminus S_\epsilon$ we have $|\de_iu_\epsilon|\le C\eta|\nabla u_\epsilon|\le C\eta F(\nabla u_\epsilon)$ for $i<n$.
Hence, for all $i=1,\dots,n-1$ we have
\begin{equation*}
    \limsup_{\e\to0}\int_{B_{10}(0)} \epsilon|\de_iu_\epsilon|^2
\le\limsup_{\e\to0}\int_{B_{10}(0)\cap S_\epsilon} \epsilon|\de_iu_\epsilon|^2+\limsup_{\e\to0}\int_{B_{10}(0)\setminus S_\epsilon} \epsilon|\de_iu_\epsilon|^2
\le C\eta^2\int\epsilon F^2(\nabla u_\epsilon),
\end{equation*}
so that
\begin{equation}\label{par.to.zero}
    \lim_{\e\to0}\int_{B_{10}(0)} \epsilon|\de_iu_\epsilon|^2=0.
\end{equation}

Let us fix a nonnegative cut-off function $\chi\in C^1_c((-1,1))$,
equal to $1$ on $(-1/2,1/2)$.

\begin{lemma}\label{constancy}
    There exist Borel sets $E_\e\subset B_1^{n-1}(0)$ such that
    $$\mathcal{L}^{n-1}(E_\e)\to0$$
    and, for any sequence of points $p_\e\in B_1^{n-1}(0)\setminus E_\e$,
    $$\int_{\{p_\e\}\times[-1,1]}\chi(x^n)e_{\epsilon}(u_\e(p_\e,x^n))\,dx^n\to \theta.$$
\end{lemma}

\begin{proof}
    We test \eqref{div.free}
    with a vector field of the form $\chi\varphi e_j$, where
    $\varphi\in C^1_c(B_1^{n-1}(0))$
    and $j=1,\dots,n-1$.
    We get
    $$\int_{B_1^{n-1}(0)} e_{\epsilon}(u_i)\de_j(\chi\varphi)=\int_{B_1^{n-1}(0)}\epsilon\de_j u_\e
     D(F^2/2)(\nabla u_\e)[\nabla\varphi]+o(\|\nabla\varphi\|_{C^0}),$$
    where the error term comes from the ambient metric $g_\e$, which converges to the Euclidean one as $\e\to0$,
    and the error term in \eqref{div.free}, which is also negligible in the limit $\e\to0$,
    as $F^\e$ converges to the autonomous integrand $F_{x_0}$.
    By \eqref{par.to.zero}, letting
    $$f_\e(x^1,\dots,x^{n-1}):=\int_{-1}^1\chi(x^n)e_{\epsilon}(u_\e(x^1,\dots,x^n))\,dx^n$$
    be the integral on the slice, we then have
    $$\left|\int_{B_1^{n-1}(0)} f_\e\de_j\varphi\right|\le \sigma_\e\|\nabla\varphi\|_{C^0}$$
    for a vanishing sequence $\sigma_\e\to0$, as $\e\to0$.
    The conclusion then follows from Allard's strong constancy lemma \cite[Theorem 1.(4)]{allard1986integrality},
    once we recall that
    $$f_\e\,dx\rightharpoonup\theta\,dx$$
    as measures on $B_1^{n-1}(0)$.
\end{proof}

We have the following simple result in $\R^n$, for autonomous $F$.

\begin{lemma}\label{one-dim}
    Given $\Lambda>0$ and $\gamma\in(0,1)$, there exists $\delta(\Lambda,\gamma,F)>0$ such that
    either
    $$|u|\ge1-C\gamma\quad\text{on }\{0\}\times(-R,R)$$
    or the energy on the slice
    \begin{equation}\label{quantiz}
    \frac{1}{F(e_n)}\int_{\{0\}\times(-R,R)}e_1(u)\in
    \bigcup_{k\in\N\setminus\{0\}}(kc_W-Ck\gamma^2,kc_W+Ck\gamma^2)\end{equation}
    for any critical point
    $$u:B_1^{n-1}(0)\times (-R,R)\to\R$$
    of the energy $E_1$,
    provided that $R\in(0,\Lambda]$, $|u|\le1$, $F(\nabla u)^2\le W(u)+\gamma^2$,
    $$\lim_{t\to\pm R}|u(0,t)|\in[1-\gamma,1],$$
    and
    $$\int_{B_1^{n-1}(0)\times (-R,R)}\sum_{j=1}^{n-1}|\de_ju|^2\le\delta.$$
\end{lemma}

In other words, if $u$ is almost constant along the first $n-1$ coordinate
directions, then its energy is close to ($F(e_n)$ times) a multiple of the energy $c_W=\int_{-1}^1\sqrt{2W}$ of the heteroclinical solution, unless $u$ is essentially constant. In the statement, $C$ is a universal constant depending just on $W$, on which
the other constants tacitly depend as well.

\begin{proof}
    Assuming for every $\delta>0$ we have a counterexample
    $u^\delta:B_1^{n-1}(0)\times(-R_\delta,R_\delta)\to\R$,
    in the limit $\delta\to0$ (up to a subsequence) we get a one-dimensional function
    $u(x)=U(x^n)$ critical for $E_1$, defined on $B_1^{n-1}(0)\times(-R,R)$ for a limiting $R=\lim_{\delta\to0}R_\delta\in[0,\Lambda]$;
    we can assume that $R>0$ since otherwise the first conclusion for $u^\delta$ trivially holds, thanks to the Lipschitz bound on $u^\delta$ (implied by the assumption $F(\nabla u^\delta)^2\le W(u^\delta)+\gamma^2$).
    
    Note that we can pass to the limit the Euler--Lagrange equation \eqref{el.div} and also
    $$\lim_{\delta\to0}\int_{\{0\}\times(-R_\delta,R_\delta)}e_1(u^\delta)=\int_{\{0\}\times(-R,R)}e_1(u),$$
    thanks to the same Lipschitz bound and Proposition \ref{regularity}, which upgrades it to convergence in $C^{1,\alpha}_{loc}$.
    Since $u$ is one-dimensional, it is smooth and satisfies
    $$U:(-R,R)\to\R,\quad F^2(e_n)U''=W'(U).$$
    Thus, $\xi:=F^2(e_n)(U')^2/2-W(U)$ is constant; actually, it is bounded by $C\gamma^2$ in absolute value:
    this follows from the uniform convergence $u^\delta\to u$ and from equicontinuity, which give
    $$\lim_{t\to\pm R}|u(0,t)|=\lim_{\delta\to0}\lim_{t\to\pm R_\delta}|u^\delta(0,t)|\in[1-\gamma,1]$$
    and thus $\lim_{t\to\pm R}W(u)(0,t)\le C\gamma^2$, as well as the fact that
    $F(\nabla u)^2\le W(u)+\gamma^2$.
    
    The solution $U$ can be extended to all of $\R$ and, by classical theory of Newtonian systems
    (viewing $-W(U)$ as the potential energy), its extension is periodic if $\xi<0$, constant or heteroclinic if $\xi=0$ (i.e., either $U\equiv\pm1$ or it is monotone, with limits $\pm1$ at infinity),
    or divergent if $\xi>0$ (i.e., goes to $\pm\infty$ at infinity, since $U'=\pm\frac{\sqrt{2W(U)+2\xi}}{F(e_n)}$ with constant sign; recall that the right-hand side is $\simeq\pm|U|$ for $|U|\ge2$).
    Here we are assuming that $\gamma$ is small enough, so that the steady state in $(-1,1)$,
    corresponding to the maximum point for $W|_{(-1,1)}$, is not one of the possibilities.

    In the second and third cases, $U$ is monotone;
    applying the elementary inequality $0\le a+b-2\sqrt{ab}\le|a-b|$ with $a:=F^2(e_n)(U')^2/2$ and $b:=W(U)$, we see that    
    \begin{equation}\label{en.vs.jac}
    \int_{\{0\}\times(-R,R)}e_1(u)
    =\int_{-R}^{R}F(e_n)\sqrt{2W(U)}|U'|+O(\gamma^2),\end{equation}
    from which the claim follows: either $|U|\ge1-\gamma$ or \eqref{quantiz} holds with $k=1$, thanks to a simple change of variables and the boundary condition
    $|U(\pm R)|\in[1-\gamma,1]$.
    
    If instead we are in the first case where $\xi<0$, then we can divide $[-R,R]$ into
    $N\ge1$ consecutive intervals $[t_j,t_{j+1}]$ where $U$ is monotone, for $j=0,\dots, N-1$,
    such that $t_0=-R$, $t_N=R$, and
    $$U(t_j)\in W^{-1}(-\xi)\quad\text{for all }j=1,\dots,N-1$$
    (as $U'(t_j)=0$ and hence $\xi=-W(U(t_j))$ for these $j$).
    Using again \eqref{en.vs.jac} and writing
    $$\int_{-R}^{R}F(e_n)\sqrt{2W(U)}|U'|=\sum_{j=0}^{N-1}\int_{t_j}^{t_{j+1}}F(e_n)\sqrt{2W(U)}|U'|,$$
    the claim follows again from the monotonicity of $U$ on each interval,
    together with the fact that $U(t_j)\in W^{-1}(-\xi)$ forces $|U(t_j)|\in [1-C\gamma,1]$,
    as $-\xi\in[0,C\gamma^2]$ and $W$ is nearly quadratic around $\pm1$.
    Thus, the claim must also hold for $u^\delta$ eventually, a contradiction.
\end{proof}

\begin{remark}\label{rem-one-dim}
    Clearly, the same also holds on a small geodesic ball in $M$,
    rescaled to have unit size. Indeed, on such rescaled ball,
    the metric is close to the Euclidean one and $F$ is almost autonomous.
\end{remark}

We are finally ready to prove integrality of the limit varifold $V$ built in Theorem \ref{varifold.constr}, which
reduces to the following statement thanks to the blow-up reduction.

\begin{theorem}\label{thm:int}
    We have $\frac{\theta}{F_{x_0}(e_n)}\in c_W\N$.
\end{theorem}

\begin{proof}
    Let us fix $\gamma\in(0,1)$ small and recall that, since by assumption the limit of the energy densities is supported on $P$, we have
    \begin{equation}\label{stupid.bd}
        \int_{B_{10}(0)}(1-\chi(x^n))e_\e(u_\e)(x)\,dx\to0.
    \end{equation}
    Using the previous bounds, we can find $p_\e\in B_1^{n-1}(0)$ such that
    $$\int_{-1}^1 (1-\chi(t))e_\e(u_\e)(p_\e,t)\,dt\to0,$$
    thanks to \eqref{stupid.bd}, as well as
    $$\int_{-1}^1 \chi(t)e_\e(u_\e)(p_\e,t)\,dt\to\theta$$
    by virtue of Lemma \ref{constancy},
    \begin{equation}\label{resc.max.bd}
        \frac{1}{\e^{n-1}}\int_{B_\e^{n-1}(p_\epsilon)\times[-1,1]}\sum_{j=1}^{n-1}|\de_ju_\e|^2\to0
    \end{equation}
    by \eqref{par.to.zero} and the standard weak-$(1,1)$ maximal bound,
    and finally
    \begin{equation}\label{bd.en.gamma.again}
    \limsup_{\e\to0}\int_{-1}^1 1_{[1-C_0\gamma,1]}(|u_\e|(p_\e,t))e_\e(u_\e)(p_\e,t)\,dt\le C\gamma
    \end{equation}
    thanks to \eqref{bd.en.gamma.bis}, where $C_0:=C$ is the constant appearing in Lemma \ref{one-dim}.

    We write $\{t\in(-1,1)\,:\,|u_\e|(p_\e,t)<1-\gamma\}=\bigcup_{j=1}^{N_\e} I_j$
    as a union of disjoint open intervals. Since
    $W(u_\e)\ge c(\gamma)>0$ on $\{|u_\e|<1-\gamma\}$, we have
    $$\sum_j|I_j|\le C(\gamma)\int_{-1}^1 W(u_\e)(p_\e,t)\,dt\le C(\gamma)\int_{-1}^1 \e e_1(p_\e,t)\,dt\le C(\gamma)\e=:\Lambda\e.$$
    This implies that, for $\e$ small enough, $\chi(t)$ vanishes for $t\in I_0\cup I_N$;
    in order to study
    $$\theta=\lim_{\e\to0}\int_{-1}^1 \chi(t)e_\e(u_\e)(p_\e,t)\,dt=\lim_{\e\to0}\int_{-1}^1 e_\e(u_\e)(p_\e,t)\,dt,$$
    we can then look at the intermediate intervals $I_j$ for $j=1,\dots,N-1$.
    We can now apply Lemma \ref{one-dim} (and Remark \ref{rem-one-dim}) to each of the cylinders
    $$B_\e^{n-1}(p_\e)\times I_j,\quad j\in\{1,\dots,N-1\},$$
    dilated by a factor $\e^{-1}$, so that the dilated function becomes critical for $E_1$.
    The assumptions of this lemma hold eventually, thanks to the bound $\e^{-1}|I_j|\le\Lambda$, the bound \eqref{modica.blowup}, the fact that
    $$|u_\e(p_\e,t)|\in[1-\gamma,1]\quad\text{for }t\in\de I_j,$$
    and \eqref{resc.max.bd}.
    
    We deduce that either $|u_\e(p_\e,t)|\in[1-C_0\gamma,1]$ for all $t\in I_j$ or
    $$\frac{1}{F_{x_0}(e_n)}\int_{I_j}e_\e(u_\e)(p_\e,t)\in\bigcup_{k\in\N\setminus\{0\}}(kc_W-Ck\gamma^2,kc_W+Ck\gamma^2).$$
    In particular, in the second case we have $\frac{1}{F_{x_0}(e_n)}\int_{I_j}e_\e(u_\e)(p_\e,t)\ge c_W/2$
    (provided $\gamma$ is small enough). Since the energy on the full slice $\{p_\e\}\times[-1,1]$
    is at most $2\theta$, we deduce that there is an upper bound $C$ for the number of such intervals
    of the second kind,
    independent of $\gamma$.

    Recalling \eqref{bd.en.gamma.again}, we arrive at
    $$\frac{1}{F_{x_0}(e_n)}\int_{-1}^1e_\e(u_\e)(p_\e,t)\,dt
    =\frac{1}{F_{x_0}(e_n)}\int_{[-1,1]\setminus(I_0\cup I_N)}e_\e(u_\e)(p_\e,t)\,dt=n_\e c_W+O(\gamma),$$
    for some $n_\e\in\N$, with an implied constant independent of $\e,\gamma$.
    Since
    $$\frac{\theta}{F_{x_0}(e_n)}=\lim_{\e\to0}\frac{1}{F_{x_0}(e_n)}\int_{-1}^1e_\e(u_\e)(p_\e,t)\,dt=\lim_{\e\to0}n_\e c_W+O(\gamma),$$
    we obtain that the distance of $\frac{\theta}{F_{x_0}(e_n)}$ from $c_W\N$ is bounded by $O(\gamma)$.
    Since $\gamma$ was arbitrary, we reach the conclusion.
\end{proof}

\bibliographystyle{siam}
\bibliography{biblio}

@article{ModicaMortola1977,
  author    = {Luciano Modica and Stefano Mortola},
  title     = {Un esempio di $\Gamma$-convergenza},
  journal   = {Bollettino della Unione Matematica Italiana B},
  volume    = {14},
  number    = {1},
  pages     = {285--299},
  year      = {1977}
}

@article{Guaraco2018,
  author    = {Marcos M. Guaraco},
  title     = {Min-max for phase transitions and the existence of embedded minimal hypersurfaces},
  journal   = {J. Diff. Geom.},
  volume    = {108},
  number    = {1},
  pages     = {91--133},
  year      = {2018}
}

@article{Almgren1965,
  author    = {Almgren, Jr., F. J.},
  title     = {The theory of varifolds: A variational calculus in the large for the $k$-dimensional area integrand},
  journal   = {Bull. Amer. Math. Soc.},
  volume    = {71},
  pages     = {803--804},
  year      = {1965}
}

@book{Pitts1981,
  author    = {Jon T. Pitts},
  title     = {Existence and regularity of minimal surfaces on Riemannian manifolds},
  series    = {Annals of Mathematics Studies},
  volume    = {27},
  publisher = {Princeton University Press},
  year      = {1981}
}

@article{MarquesNeves2014,
  author    = {Fernando C. Marques and Andr{\'e} Neves},
  title     = {The {W}illmore conjecture},
  journal   = {Ann. of Math. (2)},
  volume    = {179},
  number    = {2},
  pages     = {683--782},
  year      = {2014}
}

@article{MarquesNeves2017,
  author    = {Fernando C. Marques and Andr{\'e} Neves},
  title     = {Existence of infinitely many minimal hypersurfaces in positive {R}icci curvature},
  journal   = {Invent. Math.},
  volume    = {209},
  pages     = {577--616},
  year      = {2017}
}

@article{Song2023,
  author    = {Antoine Song},
  title     = {Existence of infinitely many minimal hypersurfaces in closed manifolds},
  journal   = {Ann. of Math. (2)},
  volume    = {197},
  number    = {3},
  year      = {2023},
  pages     = {859--895},
  doi       = {10.4007/annals.2023.197.3.1},
  mrnumber  = {4565470}
}

@article{ChodoshMantoulidis2023,
  author    = {Otis Chodosh and Christos Mantoulidis},
  title     = {Minimal surfaces and the {Allen}--{Cahn} equation on 3-manifolds: the multiplicity one conjecture},
  journal   = {Ann. of Math. (2)},
  volume    = {197},
  number    = {3},
  pages     = {1279--1361},
  year      = {2023}
}

@article{PadillaTonegawa1998,
  author    = {P. Padilla and Y. Tonegawa},
  title     = {On the convergence of stable phase transitions},
  journal   = {Comm. Pure Appl. Math.},
  volume    = {51},
  number    = {6},
  year      = {1998},
  pages     = {551--579},
  doi       = {10.1002/(SICI)1097-0312(199806)51:6<551::AID-CPA2>3.0.CO;2-B},
  mrnumber  = {1612879}
}

@article{DRT,
  author    = {De Rosa, Antonio and  Tione, Riccardo},
  title     = {Regularity for graphs with bounded anisotropic mean curvature},
  journal   = {Invent. Math.},
  volume    = {230},
  number    = {2},
  pages     = {463--507},
  year      = {2022},
  doi       = {10.1007/s00222-022-01129-6}
}

@article{DeR-survey,
  author    = {De Rosa, Antonio},
  title     = {On the Theory of Anisotropic Minimal Surfaces},
  journal   = {Notices Amer. Math. Soc.},
  volume    = {71},
  number    = {7},
  pages     = {853--859},
  year      = {2024},
  doi       = {10.1090/noti2980},
  url       = {https://www.ams.org/journals/notices/202407/noti2980/noti2980.html}
}

@article{Alm68,
    AUTHOR = {Almgren, Jr., F. J.},
     TITLE = {Existence and regularity almost everywhere of solutions to
              elliptic variational problems among surfaces of varying
              topological type and singularity structure},
   JOURNAL = {Ann. of Math. (2)},
  FJOURNAL = {Annals of Mathematics. Second Series},
    VOLUME = {87},
      YEAR = {1968},
     PAGES = {321--391},
      ISSN = {0003-486X},
   MRCLASS = {53.04 (49.00)},
  MRNUMBER = {MR0225243 (37 \#837)},
MRREVIEWER = {E. Baiada},
}

@article {Alm76,
    AUTHOR = {Almgren, Jr., F. J.},
     TITLE = {Existence and regularity almost everywhere of solutions to
              elliptic variational problems with constraints},
   JOURNAL = {Mem. Amer. Math. Soc.},
  FJOURNAL = {Memoirs of the American Mathematical Society},
    VOLUME = {4},
      YEAR = {1976},
    NUMBER = {165},
      ISSN = {0065-9266},
   MRCLASS = {49F22 (58E15)},
  MRNUMBER = {0420406 (54 \#8420)},
MRREVIEWER = {Jean E. Taylor},
}

@article{DRR,
  title   = {Boundary regularity for anisotropic minimal {L}ipschitz graphs},
  author  = {De Rosa, Antonio and Resende, Reinaldo},
  journal = {Comm. Partial Differential Equations},
  volume  = {49},
  number  = {1-2},
  pages   = {15--37},
  year    = {2024},
  doi     = {10.1080/03605302.2023.2294335},
  url     = {https://doi.org/10.1080/03605302.2023.2294335}
}

@article{DPDRG2,
  title   = {Existence Results for Minimizers of Parametric Elliptic Functionals},
  author  = {De Philippis, Guido and De Rosa, Antonio and Ghiraldin, Francesco},
  journal = {J. Geom. Anal.},
  volume  = {30},
  number  = {2},
  pages   = {1450--1465},
  year    = {2020},
  doi     = {10.1007/s12220-019-00165-8},
  url     = {https://doi.org/10.1007/s12220-019-00165-8}
}

@article{DLDRG,
  title   = {A direct approach to the anisotropic {P}lateau's problem},
  author  = {De Lellis, Camillo and De Rosa, Antonio and Ghiraldin, Francesco},
  journal = {Adv. Calc. Var.},
  volume  = {12},
  number  = {2},
  pages   = {211--223},
  year    = {2019},
  doi     = {10.1515/acv-2016-0057},
  url     = {https://doi.org/10.1515/acv-2016-0057}
}

@article{DR,
  title   = {Minimization of anisotropic energies in classes of rectifiable varifolds},
  author  = {De Rosa, Antonio},
  journal = {SIAM Journal on Mathematical Analysis},
  volume  = {50},
  number  = {1},
  pages   = {162--181},
  year    = {2018},
  doi     = {10.1137/17M1112479},
  url     = {https://doi.org/10.1137/17M1112479}
}

@article{AllenCahn1977,
  author    = {S. M. Allen and J. W. Cahn},
  title     = {A microscopic theory for domain wall motion and its experimental verification in {Fe--Al} alloy domain growth kinetics},
  journal   = {Le Journal de Physique Colloques},
  volume    = {38},
  year      = {1977},
  pages     = {C7.51--C7.57},
  doi       = {10.1051/jphyscol:1977708},
  note      = {Proceedings of the International Conference on Solid-Solid Phase Transformations},
}

@article{DPDRminmax,
author = {De Philippis, Guido and De Rosa, Antonio},
title = {The anisotropic min-max theory: Existence of anisotropic minimal and {CMC} surfaces},
journal = {Comm. Pure and Appl. Math.},
volume = {77},
number = {7},
pages = {3184-3226},
doi = {https://doi.org/10.1002/cpa.22189},
url = {https://onlinelibrary.wiley.com/doi/abs/10.1002/cpa.22189},
eprint = {https://onlinelibrary.wiley.com/doi/pdf/10.1002/cpa.22189},
year = {2024}
}

@article{Modica1987,
  author    = {L. Modica},
  title     = {The gradient theory of phase transitions and the minimal interface criterion},
  journal   = {Arch. Rational Mech. Anal.},
  volume    = {98},
  year      = {1987},
  pages     = {123--142},
  doi       = {10.1007/BF00251230}
}

@article{Sternberg1988,
  author    = {P. Sternberg},
  title     = {The effect of a singular perturbation on nonconvex variational problems},
  journal   = {Arch. Rational Mech. Anal.},
  volume    = {101},
  year      = {1988},
  pages     = {209--260},
  doi       = {10.1007/BF00251589}
}

@incollection{Pacard2012,
  author    = {F. Pacard},
  title     = {The role of minimal surfaces in the study of the {Allen--Cahn} equation},
  booktitle = {Geometric Analysis: Partial Differential Equations and Surfaces, UIMP--RSME Santalo Summer School, University of Granada, June 28--July 2, 2010},
  series    = {Contemporary Mathematics},
  volume    = {570},
  pages     = {137--163},
  year      = {2012},
  publisher = {American Mathematical Society},
  address   = {Providence, RI},
  doi       = {10.1090/conm/570/11308}
}

@incollection{Savin2009,
  author    = {O. Savin},
  title     = {Phase transitions, minimal surfaces and a conjecture of {De Giorgi}},
  booktitle = {Current Developments in Mathematics, 2009},
  pages     = {59--113},
  year      = {2010},
  publisher = {International Press},
  address   = {Somerville, MA}
}

@article{Tonegawa2008,
  author    = {Yoshihiro Tonegawa},
  title     = {Applications of geometric measure theory to two-phase separation problems},
  journal   = {Sugaku Expositions},
  volume    = {21},
  number    = {1},
  year      = {2008},
  pages     = {97--115},
  mrnumber  = {2406276}
}

@article{HutchinsonTonegawa2000,
  author    = {J. E. Hutchinson and Y. Tonegawa},
  title     = {Convergence of phase interfaces in the van der {Waals--Cahn--Hilliard} theory},
  journal   = {Calc. Var. Partial Differential Equations},
  volume    = {10},
  year      = {2000},
  pages     = {49--84},
  doi       = {10.1007/s005260050006},
  mrnumber  = {1760560}
}

@article{Tonegawa2005,
  author    = {Y. Tonegawa},
  title     = {On stable critical points for a singular perturbation problem},
  journal   = {Comm. Anal. Geom.},
  volume    = {13},
  year      = {2005},
  pages     = {439--459},
  doi       = {10.4310/CAG.2005.v13.n2.a5},
  mrnumber  = {2154829}
}

@article{TonegawaWickramasekera2012,
  author    = {Y. Tonegawa and N. Wickramasekera},
  title     = {Stable phase interfaces in the van der {Waals--Cahn--Hilliard} theory},
  journal   = {J. Reine Angew. Math.},
  volume    = {2012},
  number    = {668},
  year      = {2012},
  pages     = {191--210},
  doi       = {10.1515/CRELLE.2011.093},
  mrnumber  = {2948876}
}

@article{Wickramasekera2014,
  author    = {N. Wickramasekera},
  title     = {A general regularity theory for stable codimension 1 integral varifolds},
  journal   = {Ann. of Math. (2)},
  volume    = {179},
  year      = {2014},
  pages     = {843--1007},
  doi       = {10.4007/annals.2014.179.3.3},
  mrnumber  = {3171756}
}

@misc{DePhilippisDeRosaLi2025,
  title={Existence and regularity of min-max anisotropic minimal hypersurfaces},
  author={De Philippis, Guido and De Rosa, Antonio and Li, Yangyang},
  howpublished={arXiv preprint 2409.15232},
  year={2024}
}

@article{ElliottSchatzle1996,
  author  = {Elliott, Charles M. and Sch{\"a}tzle, Reiner},
  title   = {The Limit of the Anisotropic Double-Obstacle {A}llen--{C}ahn Equation},
  journal = {Proc. Royal Soc. Edinburgh},
  volume  = {126},
  number  = {6},
  pages   = {1217--1234},
  year    = {1996},
  doi     = {10.1017/S0308210500023374},
}

@article{MatanoMoriNara2019,
  author  = {Matano, Hiroshi and Mori, Yoichiro and Nara, Mitsunori},
  title   = {Asymptotic Behavior of Spreading Fronts in the Anisotropic {A}llen--{C}ahn Equation on $\mathbb{R}^n$},
  journal = {Ann. Inst. Poincar{\'e}},
  volume  = {36},
  number  = {3},
  pages   = {585--626},
  year    = {2019},
  doi     = {10.1016/j.anihpc.2018.07.008},
}

@article{ElliottSchatzle1997,
  author  = {Elliott, Charles M. and Sch{\"a}tzle, Reiner},
  title   = {The Limit of the Fully Anisotropic Double-Obstacle {A}llen--{C}ahn Equation in the Nonsmooth Case},
  journal = {SIAM J. Math. Anal.},
  volume  = {28},
  number  = {2},
  pages   = {274--303},
  year    = {1997},
  month   = mar,
  doi     = {10.1137/S0036141095289602},
}

@article{GigaOhtsukaSchatzle2006,
  author  = {Giga, Yoshikazu and Ohtsuka, Takeshi and Sch{\"a}tzle, Reiner},
  title   = {On a uniform approximation of motion by anisotropic curvature by the {A}llen--{C}ahn equations},
  journal = {Interfaces and Free Boundaries},
  volume  = {8},
  number  = {3},
  pages   = {317--348},
  year    = {2006},
  doi     = {10.4171/IFB/135},
}

@article{Bouchitte1990,
  author       = {Bouchitt{\'e}, Guy},
  title        = {Singular perturbations of variational problems arising from a two-phase transition model},
  journal      = {Applied Mathematics and Optimization},
  volume       = {21},
  number       = {3},
  pages        = {289--314},
  year         = {1990}
}

@inproceedings{Laux2020,
  author       = {Laux, Tim},
  title        = {A gradient-flow approach for the convergence of the anisotropic {A}llen--{C}ahn equation},
  booktitle    = {Geometric Aspects of Solutions to Partial Differential Equations},
  pages        = {32--34},
  series       = {RIMS K\^oky\^uroku},
  volume       = {2172},
  publisher    = {Research Institute for Mathematical Sciences, Kyoto University},
  address      = {Kyoto, Japan},
  year         = {2020}
}

@article{CicaleseNagasePisante2010,
  author       = {Cicalese, Marco and Nagase, Yuko and Pisante, Giovanni},
  title        = {The {G}ibbs--{T}homson relation for non homogeneous anisotropic phase transitions},
  journal      = {Adv. Calc. Var.},
  volume       = {3},
  number       = {3},
  year         = {2010},
  pages        = {321--344},
  doi          = {10.1515/acv.2010.014}
}

@article{ChoksiFonsecaLinVenkatraman2022,
  author    = {R. Choksi and I. Fonseca and J. Lin and R. Venkatraman},
  title     = {Anisotropic surface tensions for phase transitions in periodic media},
  journal   = {Calc. Var. Partial Differential Equations},
  volume    = {61},
  number    = {3},
  pages     = {107},
  year      = {2022},
  doi       = {10.1007/s00526-022-02193-7}
}

@article{CristoferiFonsecaHagertyPopovici2019,
  author    = {R. Cristoferi and I. Fonseca and A. Hagerty and C. Popovici},
  title     = {A homogenization result in the gradient theory of phase transitions},
  journal   = {Interfaces and Free Boundaries},
  volume    = {21},
  number    = {3},
  pages     = {367--408},
  year      = {2019},
  doi       = {10.4171/IFB/430}
}

@book{Simon,
 author = {Simon, Leon},
 title = {Lectures on geometric measure theory},
 fseries = {Proceedings of the Centre for Mathematical Analysis, Australian National University},
 series = {Proc. Cent. Math. Anal. Aust. Natl. Univ.},
 volume = {3},
 year = {1983},
 publisher = {Australian National University, Centre for Mathematical Analysis, Canberra},
 language = {English},
 zbMATH = {3870089},
 Zbl = {0546.49019}
}

@book{Ghoussoub,
 author = {Ghoussoub, Nassif},
 title = {Duality and perturbation methods in critical point theory},
 fseries = {Cambridge Tracts in Mathematics},
 series = {Camb. Tracts Math.},
 issn = {0950-6284},
 volume = {107},
 isbn = {0-521-44025-4},
 year = {1993},
 publisher = {Cambridge: Cambridge University Press},
 language = {English},
 zbMATH = {468521},
 Zbl = {0790.58002}
}

@book{AFP,
 author = {Ambrosio, Luigi and Fusco, Nicola and Pallara, Diego},
 title = {Functions of bounded variation and free discontinuity problems},
 fseries = {Oxford Mathematical Monographs},
 series = {Oxford Math. Monogr.},
 isbn = {0-19-850245-1},
 year = {2000},
 publisher = {Oxford: Clarendon Press},
 language = {English},
 zbMATH = {1448982},
 Zbl = {0957.49001}
}

@article{Ilmanen,
 author = {Ilmanen, Tom},
 title = {Convergence of the {Allen}--{Cahn} equation to {Brakke}'s motion by mean curvature},
 fjournal = {Journal of Differential Geometry},
 journal = {J. Diff. Geom.},
 issn = {0022-040X},
 volume = {38},
 number = {2},
 pages = {417--461},
 year = {1993},
 language = {English},
 doi = {10.4310/jdg/1214454300},
 zbMATH = {488229},
 Zbl = {0784.53035}
}

@book{LU,
 author = {Ladyzhenskaya, O. A. and Ural'tseva, N. N.},
 title = {Linear and quasilinear elliptic equations},
 fseries = {Mathematics in Science and Engineering},
 series = {Math. Sci. Eng.},
 volume = {46},
 year = {1968},
 publisher = {Elsevier, Amsterdam},
 language = {English},
 zbMATH = {3262653},
 Zbl = {0164.13002}
}

@article{Modica,
 author = {Modica, Luciano},
 title = {A gradient bound and a {Liouville} theorem for nonlinear {Poisson} equations},
 fjournal = {Communications on Pure and Applied Mathematics},
 journal = {Comm. Pure Appl. Math.},
 issn = {0010-3640},
 volume = {38},
 pages = {679--684},
 year = {1985},
 language = {English},
 doi = {10.1002/cpa.3160380515},
 zbMATH = {3989913},
 Zbl = {0612.35051}
}

@article{FeldmanMorfe2023,
  author  = {Feldman, William and Morfe, Peter},
  title   = {The Occurrence of Surface Tension Gradient Discontinuities and Zero Mobility for {A}llen--{C}ahn and Curvature Flows in Periodic Media},
  journal = {Interfaces and Free Boundaries},
  volume  = {25},
  number  = {4},
  pages   = {567--631},
  year    = {2023},
  doi     = {10.4171/IFB/504},
}

@inproceedings{allard1986integrality,
  title={An integrality theorem and a regularity theorem for surfaces whose first variation with respect to a parametric elliptic integrand is controlled},
  author={Allard, William K},
  booktitle={Proc. Symp. Pure Math},
  volume={44},
  pages={1--28},
  year={1986}
}

@article{AllardStab,
	title        = {An a priori estimate for the oscillation of the normal to a hypersurface whose first and second variation with respect to an elliptic integrand is controlled},
	author       = {Allard, William  K.},
	year         = 1983,
	journal      = {Invent. Math.},
	volume       = 73,
	number       = 2,
	pages        = {287--331},
	doi          = {10.1007/bf01394028},
	issn         = {0020-9910,1432-1297},
	url          = {https://doi.org/10.1007/BF01394028},
	fjournal     = {Inventiones Mathematicae},
	mrclass      = {49f20 (58e15)},
	mrnumber     = 714094,
	mrreviewer   = {J.\ E.\ Brothers}
}

@inproceedings{characterizationareaallard,
  title={A characterization of the area integrand},
  author={Allard, William K},
  booktitle={Symposia Mathematica},
  volume={14},
  pages={429--444},
  year={1974}
}

@article{de2018rectifiability,
  title={Rectifiability of varifolds with locally bounded first variation with respect to anisotropic surface energies},
  author={De Philippis, Guido and De Rosa, Antonio and Ghiraldin, Francesco},
  journal={Comm. Pure Appl. Math.},
  volume={71},
  number={6},
  pages={1123--1148},
  year={2018},
  publisher={Wiley Online Library}
}

@article{allard1972first,
  title={On the first variation of a varifold},
  author={Allard, William K},
  journal={Annals of mathematics},
  volume={95},
  number={3},
  pages={417--491},
  year={1972},
  publisher={JSTOR}
}

@article{LRGamma,
 author = {Lin, Fanghua and Rivi{\`e}re, Tristan},
 title = {Complex {Ginzburg}--{Landau} equations in high dimensions and codimension two area minimizing currents},
 fjournal = {Journal of the European Mathematical Society (JEMS)},
 journal = {J. Eur. Math. Soc. (JEMS)},
 issn = {1435-9855},
 volume = {1},
 number = {3},
 pages = {237--311},
 year = {1999},
 language = {English},
 doi = {10.1007/s100970050008},
 zbMATH = {1398351},
 Zbl = {0939.35056}
}

@article{JS,
 author = {Jerrard, Robert L. and Soner, Halil Mete},
 title = {The {Jacobian} and the {Ginzburg}--{Landau} energy},
 fjournal = {Calculus of Variations and Partial Differential Equations},
 journal = {Calc. Var. Partial Differential Equations},
 issn = {0944-2669},
 volume = {14},
 number = {2},
 pages = {151--191},
 year = {2002},
 language = {English},
 doi = {10.1007/s005260100093},
 zbMATH = {1756808},
 Zbl = {1034.35025}
}

@article{ABO,
 author = {Alberti, G. and Baldo, S. and Orlandi, G.},
 title = {Variational convergence for functionals of {Ginzburg}--{Landau} type},
 fjournal = {Indiana University Mathematics Journal},
 journal = {Indiana Univ. Math. J.},
 issn = {0022-2518},
 volume = {54},
 number = {5},
 pages = {1411--1472},
 year = {2005},
 language = {English},
 doi = {10.1512/iumj.2005.54.2601},
 url = {hdl.handle.net/11568/92589},
 zbMATH = {2246719},
 Zbl = {1160.35013}
}

@article{LR,
 author = {Lin, Fanghua and Rivi{\`e}re, Tristan},
 title = {A quantization property for static {Ginzburg}--{Landau} vortices.},
 fjournal = {Communications on Pure and Applied Mathematics},
 journal = {Comm. Pure Appl. Math.},
 issn = {0010-3640},
 volume = {54},
 number = {2},
 pages = {206--228},
 year = {2001},
 language = {English},
 doi = {10.1002/1097-0312(200102)54:2<206::AID-CPA3>3.0.CO;2-W},
 zbMATH = {1592012},
 Zbl = {1033.58013}
}

@article{BBO,
 author = {Bethuel, F. and Brezis, H. and Orlandi, G.},
 title = {Asymptotics for the {Ginzburg}--{Landau} equation in arbitrary dimensions},
 fjournal = {Journal of Functional Analysis},
 journal = {J. Funct. Anal.},
 issn = {0022-1236},
 volume = {186},
 number = {2},
 pages = {432--520},
 year = {2001},
 language = {English},
 doi = {10.1006/jfan.2001.3791},
 zbMATH = {1703083},
 Zbl = {1077.35047}
}

@article{SternThesis,
 author = {Stern, Daniel},
 title = {Existence and limiting behavior of min-max solutions of the {Ginzburg}--{Landau} equations on compact manifolds},
 fjournal = {Journal of Differential Geometry},
 journal = {J. Diff. Geom.},
 issn = {0022-040X},
 volume = {118},
 number = {2},
 pages = {335--371},
 year = {2021},
 language = {English},
 doi = {10.4310/jdg/1622743143},
 zbMATH = {7367830},
 Zbl = {1472.35368}
}

@article{ChengThesis,
 author = {Cheng, Da Rong},
 title = {Asymptotics for the {Ginzburg}--{Landau} equation on manifolds with boundary under homogeneous {Neumann} condition},
 fjournal = {Journal of Functional Analysis},
 journal = {J. Funct. Anal.},
 issn = {0022-1236},
 volume = {278},
 number = {4},
 note = {No. 108364},
 year = {2020},
 language = {English},
 doi = {10.1016/j.jfa.2019.108364},
 zbMATH = {7145954},
 Zbl = {1433.35375}
}

@article{PigSte,
 author = {Pigati, Alessandro and Stern, Daniel},
 title = {Minimal submanifolds from the abelian {Higgs} model},
 fjournal = {Inventiones Mathematicae},
 journal = {Invent. Math.},
 issn = {0020-9910},
 volume = {223},
 number = {3},
 pages = {1027--1095},
 year = {2021},
 language = {English},
 doi = {10.1007/s00222-020-01000-6},
 zbMATH = {7330736},
 Zbl = {1467.53025}
}

@article{PPSGamma,
 author = {Parise, Davide and Pigati, Alessandro and Stern, Daniel},
 title = {Convergence of the self-dual {{\(U(1)\)}}-{Yang}--{Mills}--{Higgs} energies to the {{\((n-2)\)}}-area functional},
 fjournal = {Communications on Pure and Applied Mathematics},
 journal = {Comm. Pure Appl. Math.},
 issn = {0010-3640},
 volume = {77},
 number = {1},
 pages = {670--730},
 year = {2024},
 language = {English},
 doi = {10.1002/cpa.22150},
 zbMATH = {7782037},
 Zbl = {1536.81007}
}

@article{SternS1,
 author = {Stern, Daniel},
 title = {{{\(p\)}}-harmonic maps to {{\(S^1\)}} and stationary varifolds of codimension two},
 fjournal = {Calculus of Variations and Partial Differential Equations},
 journal = {Calc. Var. Partial Differential Equations},
 issn = {0944-2669},
 volume = {59},
 number = {6},
 note = {No. 187},
 year = {2020},
 language = {English},
 doi = {10.1007/s00526-020-01859-6},
 zbMATH = {7294608},
 Zbl = {1475.58009}
}

@misc{PPS,
 author = {Parise, Davide and Pigati, Alessandro and Stern, Daniel},
 title = {Nonabelian {Yang}-{Mills}-{Higgs} and {Plateau}'s problem in codimension three},
 howpublished = {arXiv preprint 2502.07756},
 year = {2025},
 url = {https://arxiv.org/abs/2502.07756},
 arXiv = {arXiv:2502.07756}
}

\vspace{0.5cm}

\small{
Antonio De Rosa,\\
\emph{Department of Decision Sciences and BIDSA, Bocconi University, Milan, Italy,}\\
\emph{Email address:} antonio.derosa@unibocconi.it,}

\vspace{0.2cm}

\small{
Alessandro Pigati,\\
\emph{Department of Decision Sciences and BIDSA, Bocconi University, Milan, Italy,}\\
\emph{Email address:} 
alessandro.pigati@unibocconi.it}

\end{document}